\documentclass[12pt]{amsart}
\usepackage[margin=1.25in]{geometry}
\usepackage{amsfonts, amssymb, amsthm}
\usepackage[mathscr]{eucal}
\usepackage{array}
\usepackage[all]{xypic}

\usepackage{tikz}
\definecolor{untgreen}{RGB}{5,144,51}
\definecolor{cublue}{RGB}{75,146,219}
\definecolor{cugold}{RGB}{207,184,124}
\usepackage{aliascnt}
\usepackage[ocgcolorlinks, breaklinks=true, pagebackref, 
allcolors=untgreen]{hyperref} 
\usepackage{breakurl}

\makeatletter
\renewcommand\subsection{\@startsection{subsection}{2}%
  \z@{.35\linespacing\@plus.7\linespacing}{.25\linespacing}%
  {\normalfont\itshape}}
\makeatother

\numberwithin{equation}{section}

\theoremstyle{plain}
\newtheorem{lemma}{Lemma}[section]

\newaliascnt{proposition}{lemma}
\newtheorem{proposition}[proposition]{Proposition}
\aliascntresetthe{proposition} 

\newaliascnt{theorem}{lemma}
\newtheorem{theorem}[theorem]{Theorem} 
\aliascntresetthe{theorem}
\newtheorem*{theorem*}{Theorem}

\newaliascnt{corollary}{lemma}
\newtheorem{corollary}[corollary]{Corollary} 
\aliascntresetthe{corollary}

\newtheoremstyle{citing}
  {3pt}
  {3pt}
  {\itshape}
  {}
  {\bfseries}
  {.}
  {.5em}
  {\thmnote{#3}}

\theoremstyle{citing}

\def\equationautorefname~#1\null{(#1)\null}
\def\itemautorefname~#1\null{(#1)\null}
\def\sectionautorefname~#1\null{\S#1\null}
\def\subsectionautorefname~#1\null{\S#1\null}

 \newcommand\CE{{\mathcal E}}
\newcommand\CF{{\mathcal F}} \newcommand\CG{{\mathcal G}}
\newcommand\CH{{\mathcal H}} \newcommand\CL{{\mathcal L}}
 
\newcommand\CO{{\mathcal O}}

\newcommand\fb{{\mathfrak b}} 
\newcommand\fg{{\mathfrak g}} \newcommand\fh{{\mathfrak h}}
\newcommand\fp{{\mathfrak p}} 
\newcommand\fu{{\mathfrak u}} \newcommand\fl{{\mathfrak l}}
 
\newcommand\FN{{\mathfrak N}} 
 \newcommand\fv{{\mathfrak v}}

\newcommand\BBC{{\mathbb C}} 
\newcommand\BBF{{\mathbb F}} \newcommand\BBP{{\mathbb P}}
\newcommand\BBQ{{\mathbb Q}} 
\newcommand\BBZ{{\mathbb Z}}

 \newcommand\Cbar{\overline{C}}
 \newcommand\Gbar{\overline{G}}
\newcommand\Hbar{\overline{H}} 
\newcommand\ktilde{{\widetilde{k}}} \newcommand\Lbar{\overline{L}}
\newcommand\ptilde{{\tilde {p}}} \newcommand\Tbar{\overline{T}}
\newcommand\Xtilde{\widetilde X}
\newcommand\Ytilde{\widetilde Y} \newcommand\Ztilde{{\widetilde Z}}

\DeclareMathOperator{\Ad}{Ad} 

\DeclareMathOperator{\Aut}{Aut} 
\DeclareMathOperator{\gr}{gr}
\DeclareMathOperator{\Hom}{Hom} 
\DeclareMathOperator{\Lie}{Lie}
\newcommand{\red}{\operatorname{r}} 
\DeclareMathOperator{\res}{res}

\newcommand\id{{id}}
\newcommand\inverse{^{-1}}

\newcommand\CHIJ{{\mathcal H}^{IJ}}
\newcommand\etaIJ{\eta^{IJ}}
\newcommand\etaEJ{\eta^{\emptyset J}}
\newcommand\FNt{\widetilde{\FN}}
\newcommand\IzJ{I\cap{}^zJ}
\newcommand\LIbar{\overline{L}_I}
\newcommand\LJbar{\overline{L}_J}
\newcommand\OIbar{{\overline{\CO}_I}} 
\newcommand\Osbar{{\overline{\CO}_{\{s\}}}} 

\newcommand\SN{\mathfrak N} 
\newcommand\SNt{\widetilde{\SN}}
\newcommand\ssleq{{\scriptscriptstyle \leq}}
\newcommand\sslt{{\scriptscriptstyle <}}
\newcommand\Waf{W_{\operatorname{af}}}
\newcommand\Wex{W_{\operatorname{ex}}}
\newcommand\WIJ{{W^{IJ}}} 
\newcommand\XEJ{X^{\emptyset J}}
\newcommand\XIE{{X^{I \emptyset}}}
\newcommand\XIJ{X^{IJ}}
\newcommand\ZIbar{{\overline{Z}_I}} 
\newcommand\Zsbar{{\overline{Z}_{\{s\}}}} 

\begin{document}


\title[Equivariant K-theory] {Equivariant K-theory of
  generalized Steinberg varieties}

\author[J. M. Douglass]{J. Matthew Douglass} 
\address{Department of  Mathematics\\University of North Texas\\Denton TX,
  USA 76203}
\email{douglass@unt.edu} \urladdr{http://hilbert.math.unt.edu}

\author[G. R\"ohrle]{Gerhard R\"ohrle}
\address{Fakult\"at f\"ur Mathematik\\Ruhr-Universit\"at Bochum\\D-44780
  Bochum, Germany} 
\email{gerhard.roehrle@rub.de}
\urladdr{http://www.ruhr-uni-bochum.de/ffm/Lehrstuehle/Lehrstuhl-VI}



\begin{abstract}
  We describe the equivariant $K$-groups of a family of generalized
  Steinberg varieties that interpolates between the Steinberg variety of a
  reductive, complex algebraic group and its nilpotent cone in terms of the
  extended affine Hecke algebra and double cosets in the extended affine
  Weyl group. As an application, we use this description to define
  Kazhdan-Lusztig ``bar'' involutions and Kazhdan-Lusztig bases for these
  equivariant $K$-groups.
\end{abstract}

\maketitle

\section{Introduction}

Suppose $G(\BBF_q)$ is a Chevalley group defined over the finite field
$\BBF_q$. A fundamental result in the classification of irreducible complex
representations of $G(\BBF_q)$ is the classification of representations that
contain a vector fixed by a Borel subgroup $B(\BBF_q)$ of $G(\BBF_q)$. These
representations are completely determined, and their characters may be
computed, using the centralizer ring or Hecke algebra $\CH(G(\BBF_q),
B(\BBF_q))$. Iwahori \cite{iwahori:structure} conjectured that
$\CH(G(\BBF_q), B(\BBF_q))$ is isomorphic to the group algebra of the Weyl
group $W$ of $G(\BBF_q)$. An explicit isomorphism between $\CH(G(\BBF_q),
B(\BBF_q))$ and the group algebra of $W$ was constructed by Lusztig
\cite{lusztig:theorem}.  More generally, irreducible representations that
contain a vector fixed by a parabolic subgroup $P_I(\BBF_q)$ of $G(\BBF_q)$
are determined by the parabolic Hecke algebra $\CH(G(\BBF_q), P_I(\BBF_q))$.
Curtis, Iwahori, and Kilmoyer \cite{curtisiwahorikilmoyer:hecke} showed that
this last algebra is isomorphic to the Hecke algebra $\CH(W, W_I)$, where
$W_I$ is the corresponding parabolic subgroup of $W$, and Curtis
\cite{curtis:isomorphism} extended Lusztig's construction to obtain an
explicit isomorphism between $\CH(G(\BBF_q), P_I(\BBF_q))$ and $\CH(W,W_I)$.

Now suppose $G(\BBQ_p)$ is a Chevalley group defined over $\BBQ_p$. In this
case, one important class of representations consists of those
representations that contain a vector fixed by an Iwahori subgroup of
$G(\BBQ_p)$. These representations are again classified by a Hecke algebra,
this time the extended affine Hecke algebra $\CH$ of the complex dual group,
$\check G$.  Kazhdan and Lusztig \cite{kazhdanlusztig:langlands} construct
an isomorphism between $\CH$ and the equivariant $K$-theory of the Steinberg
variety of $\check G$. They then use this isomorphism to give a construction
of the irreducible representations of $\CH$. In this paper we extend their
construction and explicitly describe the equivariant $K$-groups of the
generalized Steinberg varieties $\XIJ$ from~\cite{douglassroehrle:geometry}
in terms of the Hecke algebra $\CH$.  When $I=J$, the subspace of $\CH$ we
consider is, up to an involution, the extension to $\CH$ of the subalgebra
$\CH(G(\BBF_q), P_I(\BBF_q))$ of $\CH(G(\BBF_q), B(\BBF_q))$.

In another direction, it follows from \cite[Theorem 5.1.3]
{douglassroehrle:homology} and \cite[Theorem 2.5]
{douglassroehrle:steinberg} that the rational Borel-Moore homology of $\XIJ$
may be computed algebraically as the space of $W_I\times W_J$-invariants in
the smash (semidirect) product of the coinvariant algebra of $W$ with the
group algebra of $W$. The results in this paper may be viewed as the
extension of this computation to the more refined level of equivariant
$K$-theory and the affine Hecke algebra.

From now on, suppose that $G$ is a connected, reductive complex algebraic
group such that the derived group of $G$ is simply connected. Set
$\fg=\Lie(G)$. For $g\in G$ and $x\in \fg$ write $gx$ instead of
$\Ad(g)(x)$, where $\Ad$ is the adjoint representation of $G$. Define $\FN$
to be the cone of nilpotent elements in $\fg$ and let $B$ be a fixed Borel
subgroup of $G$ with Lie algebra $\fb$. Then, the Steinberg variety of $G$
is the variety $Z$ of all triples $(x,gB,hB)$ in $\FN\times G/B \times G/B$
such that $g\inverse x, h\inverse x\in \fb$. Based on a construction of
Kazhdan and Lusztig \cite{kazhdanlusztig:langlands}, Chriss and Ginzburg
\cite {chrissginzburg:representation} and Lusztig \cite{lusztig:bases} have
shown that there is an algebra structure on $K^{G\times \BBC^*}(Z)$, the
$G\times \BBC^*$-equivariant $K$-group of $Z$, such that $K^{G\times
  \BBC^*}(Z)$ is isomorphic to the extended, affine Hecke algebra $\CH$
associated to $G$. Ostrik \cite{ostrik:equivariant} used this isomorphism to
describe $K^{G\times \BBC^*}(\FN)$ in terms of $\CH$ and to define a
Kazhdan-Lusztig ``bar'' involution, and a Kazhdan-Lusztig basis, of
$K^{G\times \BBC^*}(\FN)$. As indicated above, in this paper we describe the
equivariant $K$-groups of the generalized Steinberg varieties $\XIJ$ in
terms of $\CH$. These generalized Steinberg varieties interpolate between
$Z=X^{\emptyset \emptyset}$ and $\FN=X^{SS}$ ($S$ is the Coxeter generating
set for $W$ determined by $B$).  We then use our description to define
Kazhdan-Lusztig ``bar'' involutions and Kazhdan-Lusztig bases of the
equivariant $K$-groups $K^{G\times \BBC^*}(\XIJ)$.

The proof of the main theorem in this paper (\autoref{thm:main}) relies on
Ostrik's computation of $K^{G\times \BBC^*}(\FN)$.  For a generalized
Steinberg variety $\XIJ$ we use a filtration of $K^{G\times \BBC^*}(\XIJ)$
indexed by $G$-orbits in the product of two partial flag varieties.  In the
special case of $\FN$, there is a single $G$-orbit, the filtration of
$K^{G\times \BBC^*}(\FN)$ is trivial, and Ostrik has computed $K^{G\times
  \BBC^*}(\FN)$ in terms of $\CH$. In the general case, each associated
graded piece has the form $K^{L'\times \BBC^*}(\FN')$, where $L'$ is a Levi
factor of a parabolic subgroup of $G$ and $\FN'$ is the nilpotent cone in
$\Lie(L')$. Thus, each of these graded pieces is described using Ostrik's
theory.

In \autoref{sec:main} we give the basic constructions and state the main
theorem relating the extended affine Hecke algebra and the equivariant
$K$-theory of generalized Steinberg varieties. Assuming facts that are
proved in subsequent sections, \autoref{sec:proof} contains a proof
of~\autoref{thm:main}. The constructions of the ``standard basis,'' the
``bar'' involution, and the Kazhdan-Lusztig basis are given in
\autoref{sec:kl}. In \autoref{sec:kthy} we review the intersection/Tor
product construction in the form it is used in this paper. The final three
sections contain proofs of the main ingredients used in the proof
of~\autoref{thm:main}.

\subsection{Notation and conventions}
With $G$ and $B$ as above, fix a maximal torus $T$ contained in $B$. Let
$W=N_G(T)/T$ be the Weyl group of $(G,T)$, let $S$ be the set of simple
reflections in $W$ determined by the choice of $B$, and let $X(T)$ be the
character group of $T$. Then $X(T)$ is a free abelian group and $W$ acts on
$X(T)$ as group automorphisms. We use additive notation for $X(T)$ and
consider the root system $\Phi$ of $(G,T)$ as a subset of $X(T)$. The roots
corresponding to the root subgroups in $B$ determine a positive system
$\Phi^+$, and a base $\Pi$, of $\Phi$. If $s$ is in $S$, then $s=s_\alpha$
for a unique $\alpha$ in $\Pi$.

Let $H$ be a complex linear algebraic group.  We use the convention that a
lowercase fraktur letter denotes the Lie algebra of the group denoted by the
same uppercase roman letter, so for example, $\fh= \Lie(H)$. Let $R(H)$
denote the representation ring of $H$ and let $X(H)$ denote the character
group of $H$. Define $\overline H$ to be the product of $H$ with the
one-dimensional complex torus $\BBC^*$, so $\overline H= H\times
\BBC^*$. Let $v\colon \Hbar \to \BBC^*$ be the character defined by
$v(h,z)=z$ for $h\in H$ and $z\in \BBC^*$ and let $A=\BBZ[v,v\inverse]$ be
the subring of $R(\Hbar)$ generated by $v$. With this notation, there is a
natural isomorphism $R(\Hbar)\cong A\otimes_{\BBZ} R(H)$.  In particular,
\[
R(\Tbar) \cong A\otimes_{\BBZ} R(T)\cong A[X(T)]\quad\text{and} \quad
R(\Gbar)\cong A\otimes_{\BBZ} R(G).
\]

When the group $H$ acts on a quasiprojective variety $Y$, let $K^{H}(Y)$
denote the Grothendieck group of the abelian category of $H$-equivariant
coherent sheaves on $Y$. The group $K^{H}(Y)$ is naturally an $R(H)$-module.

Suppose $C$ is a closed subgroup of $H$.  For a $C$-variety $F$, let
$H\times^C F$ denote the quotient of $H\times F$ by the $C$-action given by
$c\cdot (h,y)= (hc\inverse, cy)$ for $c\in C$, $h\in H$, and $y\in F$. The
image of $(h,y)$ in $H\times^C F$ is denoted by $h*y$. The group $H$ acts on
$H\times^C F$ by left multiplication and the projection $f\colon H\times^C
F\to H/C$ given by $h*y\mapsto hC$ is a well-defined $H$-equivariant
morphism. Conversely, suppose $Y$ is an $H$-variety and $f_Y\colon Y\to H/C$
is an $H$-equivariant morphism. Set $F=f_Y\inverse(C)$. Then the map
$m\colon H\times^C F\to Y$ given by $h*y\mapsto hy$ is a well-defined
$H$-equivariant isomorphism such that $f=f_Ym$. Suppose $\Cbar$ acts on $F$.
Then $\Hbar$ acts on both $\Hbar\times^{\Cbar}F$ and $H\times^CF$, and these
varieties are canonically isomorphic $\Hbar$-varieties. It follows from work
of Thomason \cite[Proposition 6.2]{thomason:algebraic} that
$K^{\Hbar}(H\times^CF)$ is naturally isomorphic to $K^{\Cbar}(F)$, and that
if $C_{\red}$ is a reductive subgroup of $C$ such that $C\cong
C_{\operatorname{u}} \rtimes C_{\red}$, where $C_{\operatorname{u}}$ is the
unipotent radical of $C$, then $K^{\Cbar}(F)$ is isomorphic to
$K^{\Cbar_{\red}}(F)$ (see \cite[\S5.2]{chrissginzburg:representation}). Let
\[
\res_F\colon K^{\Hbar}(H\times^CF) \xrightarrow{\ \cong\ }
K^{\Cbar_{\red}}(F)
\]
denote the composition of these two isomorphisms.

Suppose $Y_1$ and $Y_2$ are $H$-varieties with $Y_1\subseteq Y_2$. To
simplify the notation, if $Y_1$ is closed in $Y_2$, then we sometimes denote
the direct image map $K^{H}(Y_1) \to K^H(Y_2)$ simply by $()_*$, and if
$Y_1$ is open in $Y_2$, then we sometimes denote the restriction map
$K^{H}(Y_2) \to K^H(Y_1)$ by $()^*$.

Unless otherwise indicated, we consider $\fg$ as a $\BBC^*$-module with the
action of $\BBC^*$ given by $z\cdot x= z^{-2}x$ for $z$ in $\BBC^*$ and $x$
in $\fg$. Then $\Gbar$ acts on $\FN$ by $(g,z) \cdot x =z^{-2}g\cdot x$. For
a subgroup $P$ of $G$, $\Gbar$ acts on $G/P$ by $(g,z)\cdot hP= gh P$.
Define
\[
\FNt=\{\, (x, gB)\in \FN\times G/B\mid g\inverse x\in \fb\,\}.
\]
As above, the Steinberg variety of $G$ is
\[
Z=\{\, (x, gB, hB) \in \FN \times G/B \times G/B \mid g\inverse x, h\inverse
x\in \fb \,\} \cong \FNt\times_{\FN} \FNt .
\]
Then $\Gbar$ acts on $\FNt$ and $Z$ via the diagonal action, and the
projections
\[
p_Z\colon Z\to \FN,\quad q_Z\colon Z\to G/B \times G/B,\quad p\colon \FNt\to
\FN, \quad\text{and} \quad q\colon \FNt\to G/B
\]
are all $\Gbar$-equivariant. 

\section{Statement of the main theorem }\label{sec:main} 

\subsection{Generalized Steinberg varieties}

For $I\subseteq S$, let $W_I=\langle I \rangle$ be the subgroup of $W$
generated by $I$ and let $P_I$ be the parabolic subgroup of $G$ that
contains $B$ such that $N_{P_I}(T)/T = W_I$. Let $U_I$ denote the unipotent
radical of $P_I$ and let $L_I$ be the Levi factor of $P_I$ that contains
$T$. Then $P_I=L_IU_I$ and $\fp_I=\fl_I + \fu_I$ are Levi decompositions of
$P_I$ and $\fp_I$, respectively.  Define $\Pi_I= \{\, \alpha\in \Pi\mid
s_\alpha\in I\,\}$ and let $\Phi_I$ be the intersection of $\Phi$ with the
span of $\Pi_I$. Then, with respect to the action of $T$, $\Phi_I$ is the
set of roots of $\fl_I$, $\Phi^+\cup \Phi_I$ is the set of roots of $\fp_I$,
and $\Phi^+\setminus \Phi_I$ is the set of roots of $\fu_I$. In the special
case when $I=\emptyset$, $P_I=B$ and we define $U = U_{\emptyset}$.

Each pair of subsets $I, J\subseteq S$ determines a \emph{generalized
  Steinberg variety}
\[
\XIJ= \{\, (x, gP_I, hP_J)\in \FN\times G/P_I \times G/P_J \mid g\inverse
x\in \fp_I,\ h\inverse x\in \fp_J\,\}
\]
(see \cite[\S2]{douglassroehrle:geometry}). Define
\[
\etaIJ\colon Z\to \XIJ \quad\text{by}\quad \etaIJ(x, gB, hB)= (x, gP_I,
hP_J).
\]
Then $\XIJ$ is a $\Gbar$-variety ($\Gbar$ acts diagonally on $\XIJ$) and
$\etaIJ$ is a surjective, proper, $\Gbar$-equivariant morphism.  Notice that
\begin{itemize}
\item if $I=J=\emptyset$, then $\XIJ=Z$ and $\etaIJ$ is the identity, and
\item if $I=J=S$, then $\XIJ \cong \FN$ and $\etaIJ$ may be identified with
  projection $p_Z\colon Z\to \FN$.
\end{itemize}

\subsection{Hecke algebras}

The \emph{Iwahori-Hecke algebra of $W$} is the $A$-algebra $\CH_S$ with
$A$-basis $T_w$, for $w$ in $W$, and multiplication satisfying
\begin{equation}
  \label{eq:std}
  \begin{cases}
    T_w T_{w'}= T_{ww'}& \text{if $\ell(ww')= \ell(w)+ \ell(w')$, and} \\
    T_s^2= v^2 T_1+(v^2-1)T_s&\text{for $s$ in $S$,}
  \end{cases}
\end{equation}
where $\ell$ is the length function on $W$ determined by $S$ and the
subscript $1$ in $T_1$ denotes the identity in $W$ (see
\cite{kazhdanlusztig:coxeter}).

The \emph{extended affine Hecke algebra of $W$} is the $A$-algebra $\CH$
with generators $T_w$, $\theta_\lambda$, for $w$ in $W$ and $\lambda$ in
$X(T)$, and multiplication satisfying
\begin{equation}
  \label{eq:bernstein}
  \begin{cases}
    T_w T_{w'}= T_{ww'}& \text{if $\ell(ww')= \ell(w)+ \ell(w')$,} \\
    T_s^2= v^2 T_1+(v^2-1)T_s&\text{for $s$ in $S$,} \\
    \theta_\lambda \theta_{\mu}= \theta_{\lambda+\mu}& \text{for
      $\lambda, \mu$ in $X(T)$,} \\
    \theta_\lambda T_s -T_s \theta_{s(\lambda)}= (v^2-1) \frac
    {\theta_\lambda-\theta_{s(\lambda)}} {1-\theta_{-\alpha}} &\text{for
      $\lambda$ in $X(T)$ and $s=s_\alpha$ in $S$, and} \\
    \theta_0=T_1 &\text{is the identity in $\CH$ .}
  \end{cases}
\end{equation}
(See \cite[\S1]{lusztig:bases}. Note that for $w$ in $W$, the generator
$T_w$ in the preceding definition is related to the generator $\tilde T_w$
in \cite[\S1]{lusztig:bases} by $\tilde T_w=v^{-\ell(w)} T_w$.)

We identify the $A$-span, in $\CH$, of $\{\,T_w\mid w\in W\,\}$ with the
Iwahori-Hecke algebra $\CH_S$, and we identify the $A$-span, in $\CH$, of
$\{\,\theta_\lambda \mid \lambda \in X(T)\,\}$ with the group algebra
$A[X(T)]$ of $X(T)$. Then $\CH_S$ and $A[X(T)]$ are subalgebras of $\CH$
that contain the identity. The center of $\CH$ is $A[X(T)]^W$ (see
\cite{lusztig:singularities}).  We identify $R(\Gbar)$ with $A[X(T)]^W$, and
hence with the center of $\CH$, via the isomorphism $R(\Gbar) \cong
A[X(T)]^W$ given by associating with a representation of $\Gbar$ its
character in $A[X(T)]$.  The map $A[X(T)]\otimes_A \CH_S\to \CH$ given by
multiplication, $\theta_\lambda\otimes T_w\mapsto \theta_\lambda T_w$, is an
$A$-module isomorphism. We call $\{\, \theta_\lambda T_w\mid \lambda\in
X(T), w\in W\,\}$ the \emph{Bernstein basis} of $\CH$ because it arises from
the Bernstein presentation~\eqref{eq:bernstein}.

For $\lambda$ in $X(T)$, let $t_\lambda$ denote translation by $\lambda$ in
$X(T)$. Then $\{\, t_\lambda\mid \lambda\in X(T)\,\}$ is a subgroup of
$\Aut(X(T))$ isomorphic to $X(T)$. Recall that $W$ acts faithfully on
$X(T)$. Define $\Wex$, the \emph{extended affine Weyl group of $\Phi$,} to
be the subgroup of $\Aut(X(T))$ generated by the image of $W$ and $\{\,
t_\lambda\mid \lambda \in X(T)\,\}$. Then $\Wex$ is isomorphic to the
semi-direct product $X(T) \rtimes W$. We frequently identify $W$ with its
image in $\Aut(X(T))$ and consider $W$ as a subgroup of $\Wex$.

The \emph{affine Weyl group of $\Phi$,} $\Waf$, is the subgroup of
$\Aut(X(T))$ generated by the image of $W$ and $\{\, t_\alpha\mid \alpha \in
\Phi\,\}$. Then $\Waf$ is a normal subgroup of $\Wex$ and there is a finite
abelian subgroup $\Gamma$ of $\Wex$ such that $\Wex= \Waf \Gamma$ and
$\Waf\cap \Gamma =1$. The group $\Waf$ is a Coxeter group with a Coxeter
generating set $S_{\textrm{af}}$ that contains $S$, and $\Gamma$ acts on
$\Waf$ as Coxeter group automorphisms preserving $S_{\textrm{af}}$. Extend
the length function $\ell$ and Bruhat order $\leq$ on $\Waf$ to $\Wex$ by
defining
\[
\ell(y\gamma)= \ell(y)\qquad\text{and} \qquad \text{$y\gamma \leq y'\gamma'$
  if and only if $\gamma=\gamma'$ and $y\leq y'$}
\] 
for $y,y'$ in $\Waf$ and $\gamma, \gamma'$ in $\Gamma$ (see~\cite[\S2]
{lusztig:singularities}).

The algebra $\CH$ has a \emph{standard basis}, $\{\, T_w\mid w\in \Wex\,\}$,
such that the relations~\eqref{eq:std} hold (see \cite[\S1]
{lusztig:bases}). The ``bar'' involution of $\CH$, $\overline{\phantom {x}
}\colon \CH \to \CH$, is the ring automorphism of $\CH$ defined by
$\overline v=v\inverse$ and $\overline{T_x}= T_{x\inverse} \inverse$ for $x$
in $\Wex$. As observed by Lusztig~\cite{lusztig:singularities}, the argument
in the proof of \cite[Theorem 1.1] {kazhdanlusztig:coxeter} can be applied
to show that for $x$ in $\Wex$, there are unique elements $C_x$ and $C_x'$
in $\CH$ such that
\begin{equation}
  \label{eq:cx}
  \begin{cases}
    \overline{C_x}= C_x  \\
    C_x=v_x\inverse T_x+ \sum_{y< x} \epsilon_y\epsilon_x v_x v_y^{-2}
    \overline{P}_{y,x} T_y
  \end{cases}
\end{equation}
and
\begin{equation}
  \label{eq:cx'}
  \begin{cases}
    \overline{C_x'}= C_x' \\
    C_x'=v_x\inverse T_x+ \sum_{y< x} v_x\inverse P_{y,x}T_y,
  \end{cases}
\end{equation}
where 
\[
\epsilon_x= (-1)^{\ell(x)}\quad\text{and}\quad v_x= v^{\ell(x)}
\]
for $x$ in $\Wex$, and $P_{y,x}$ is a polynomial in $v$ of degree at most
$\ell(x)-\ell(y) -1$. (See also~\cite[\S1.7, 1.8]{lusztig:bases}.) We call
$\{\, C_x\mid x\in \Wex\,\}$ and $\{\, C_x'\mid x\in \Wex\,\}$
\emph{Kazhdan-Lusztig bases} of $\CH$. A fundamental property of the
Kazhdan-Lusztig bases is that if $x$ is in $\Wex$ and $s,t$ are in $S$ with
$\ell(tx)<\ell(x)$ and $\ell(xs)<\ell(x)$, then
\begin{equation}
  \label{eq:cxts}
  T_tC_x= C_xT_s=-C_x\quad\text{and}\quad T_tC_x'= C_x'T_s=v^2 C_x'.
\end{equation}

For a subset $I$ of $S$ let $w_I$ be the longest element in $W_I$. Then
\[
C_{w_I}= (-v)^{\ell(w_I)} \sum_{y\in W_I} \epsilon_y v_y^{-2} T_y \quad
\text{and}\quad T_sC_{w_I}= C_{w_I}T_s=-C_{w_I}
\]
for $s$ in $I$. For subsets $I$ and $J$ of $S$ we have the Kazhdan-Lusztig
basis elements $C_{w_I}$ and $C_{w_J}$ of $\CH$. Define
\[
\CHIJ=C_{w_I} \CH C_{w_J}\quad \text{and}\quad \chi^{IJ}\colon \CH\to
\CHIJ\text{ by } \chi^{IJ}(h)= C_{w_I} h C_{w_J}.
\]
Obviously, $\CHIJ$ is an $R(\Gbar)$-submodule of $\CH$ and $\chi^{IJ}$ is a
surjective $R(\Gbar)$-module homomorphism.

\subsection{The isomorphism \texorpdfstring{$\CHIJ \cong
    K^{\Gbar}(\XIJ)$}{}}  

The group $\Gbar$ acts on the Steinberg variety $Z$ and so $K^{\Gbar}(Z)$ is
naturally an $R(\Gbar)$-module, and hence an $A$-module. Chriss and Ginzburg
\cite[\S7.2]{chrissginzburg:representation} and Lusztig \cite{lusztig:bases}
have shown that $K^{\Gbar}(Z)$ has an $A$-algebra structure such that $\CH
\cong K^{\Gbar}(Z)$. Let $\varphi\colon \CH\to K^{\Gbar}(Z)$ be the
$A$-algebra isomorphism constructed by Lusztig \cite[Theorem
8.6]{lusztig:bases}. The main result in this paper is the following theorem.

\begin{theorem}\label{thm:main}
  For each pair of subsets $(I,J)$ of $S$ there is a unique
  $R(\Gbar)$-module isomorphism
  \[
  \psi^{IJ}\colon \CHIJ \xrightarrow{\ \cong\ } K^{\Gbar}(\XIJ)
  \]
  such that the diagram
  \begin{equation}
    \label{eq:mainthm}
    \vcenter{\vbox{
        \xymatrix{\CH \ar[d]^{\chi^{IJ}} \ar[rr]^-{\varphi}_-{\cong} &&
          K^{\Gbar}(Z) \ar[d]^{\etaIJ_*} \\
          \CHIJ \ar[rr]^-{\psi^{IJ}}_-\cong && K^{\Gbar}(\XIJ) }
      }}
  \end{equation}
  commutes.
\end{theorem}

In~\autoref{sec:kl} we use the isomorphism $\psi^{IJ}$ to define a standard
basis, a ``bar'' involution, and a Kazhdan-Lusztig basis for
$K^{\Gbar}(\XIJ)$.  As we explain next, the proof of~\autoref{thm:main}
gives significantly more information about the isomorphism $\psi^{IJ}$ than
is encoded in~\eqref{eq:mainthm}.  This leads to a graded refinement of the
theorem that is given in~\autoref{cor:ref}.

The preimages of the $G$-orbits on $G/P_I \times G/P_J$ under the projection
on the second and third factors, say $\{\XIJ_z\}$, form a partition of
$\XIJ$ into locally closed, equidimensional subvarieties indexed by
$\{W_IzW_J\}$, the set of $(W_I, W_J)$-double cosets in $W$. The closures of
these subvarieties are the irreducible components of $\XIJ$ (see
\cite[\S3]{douglassroehrle:geometry}). Thus, the fundamental classes of the
closures of the subvarieties $\XIJ_z$ form a basis of the top Borel-Moore
homology of $\XIJ$. As explained in detail below, the contribution of each
subvariety $\XIJ_z$ to $K^{\Gbar}(\XIJ)$ is not just a single homology
class, rather it is the full equivariant $K$-group of isomorphism classes of
$\LIbar\cap {}^z \LJbar$-equivariant coherent sheaves on the nilpotent cone
of $\fl_I\cap z\fl_J$. A basis of this $K$-group is indexed by the set of
$(W_I,W_J)$-double cosets in $W_I\backslash \Wex/W_J$ that project to the
double coset $W_IzW_J$ in $W$. Taking the union over $z$ (in a suitable
sense) gives rise to a basis of $K^{\Gbar}(\XIJ)$ indexed by the
$(W_I,W_J)$-double cosets in $\Wex$.

In case $I=J=\emptyset$, $X^{\emptyset \emptyset}=Z$, and $\{W_\emptyset z
W_\emptyset\}=W$. For $w$ in $W$, $X^{\emptyset \emptyset}_w=Z_w$ is the
conormal bundle to the $G$-orbit in $G/B\times G/B$ corresponding to $w$,
$L_\emptyset \cap {}^w L_\emptyset =T$, the cone of nilpotent elements in
$\fl_\emptyset \cap w\fl_\emptyset$ is $\{0\}$, and $K^{\Tbar}(\{0\}) \cong
A[X(T)]$. Obviously, $\{\, t_\lambda w\mid \lambda\in X(T)\,\}$ parametrizes
the set of $(W_\emptyset, W_\emptyset)$-double cosets in $\Wex$ that project
to $\{w\}$. It follows from~\autoref{thm:ost} that $\{\, t_\lambda w
\mid \lambda\in X(T)\,\}$ parametrizes an $A$-basis of $K^{\Gbar}(Z_w)$ and
it follows from ~\cite[\S8.6]{lusztig:bases} that $\{\, t_\lambda w \mid
w\in W,\ \lambda\in X(T)\,\}$ parametrizes an $A$-basis of $K^{\Gbar}(Z)$.

At the other extreme, when $I=J=S$, there is a single $(W,W)$-double coset
in $W$ (with representative $1$) and $\{W_S z W_S\}=\{W\}$. In this case,
$X^{SS}_1=X^{SS}\cong \FN$, $L_S \cap {}^1 L_S =G$, and the cone of
nilpotent elements in $\fl_S \cap {}^1\fl_S$ is $\FN$. The
$(W_S,W_S)$-double cosets in $\Wex$ that project to $W$ are simply the
$(W,W)$-double cosets in $\Wex$. These are parametrized by $\{\,
t_\lambda\mid \lambda\in X(T)^+\,\}$, where $X(T)^+$ is the set of dominant
weights relative to the choice of $B$.  For $\lambda$ in $X(T)$ let
$\CL_\lambda$ be the $G$-equivariant line bundle on $G/B$ such that $T$ acts
on the fibre over $B$ with character $-\lambda$. We consider $\CL_\lambda$
as a $\Gbar$-equivariant line bundle on $G/B$ via the natural projection
$\Gbar \to G$. It follows from a result of Broer~\cite{broer:line} that if
$\lambda$ is dominant, then $p_* q^* ([\CL_\lambda]) = [R^0p_*q^*
\CL_\lambda]$ in $K^{\Gbar}(\FN)$.  Ostrik~\cite[\S2.2] {ostrik:equivariant}
has proved the following key lemma he attributes to R.~Bezrukavnikov.

\begin{lemma}\label{lem:ostlem}
  The set $\{\, p_*q^* ([\CL_\lambda]) \mid \lambda\in X(T)^+ \,\}$ is an
  $A$-basis of $K^{\Gbar}(\FN)$.
\end{lemma}

In~\cite{ostrik:equivariant} Ostrik assumes that the group $G$ is
simple. The extension to reductive groups is straightforward. It follows
from the lemma that $\{\,t_\lambda \mid \lambda\in X(T)^+\,\}$ parametrizes
an $A$-basis of $K^{\Gbar} (X^{SS}_1)= K^{\Gbar}(\FN)$.

For arbitrary $I,J\subseteq S$, the results we prove below are an amalgam of
the two extreme cases. Let $\WIJ$ denote the set of minimal length
$(W_I,W_J)$-double coset representatives in $W$. For $z$ in $\WIJ$, set
\[
L_z=L_I\cap {}^zL_J
\]
and let $X(T)^+_{z}$ be the set of weights in $X(T)$ that are dominant for
$L_z$. Then $L_z$ is a reductive group (see~\cite[\S69B]
{curtisreiner:methodsII}). We show in~\autoref{thm:ost} that there is an
isomorphism $K^{\Gbar}(\XIJ_z)\cong K^{\Lbar_z}(\SN_z)$, where $\SN_z$ is
the nilpotent cone in $\fl_z$ and we show in~\autoref{lem:double} that $\{\,
t_\lambda z \mid \lambda\in X(T)^+_{z}\,\}$ parametrizes the set of $(W_I,
W_J)$-double cosets in $\Wex$ that project to $W_IzW_J$. It follows
from~\autoref{thm:wideiso2} that $\{\, t_\lambda z \mid \lambda\in
gX(T)^+_{z}\,\}$ parametrizes an $A$-basis of $K^{\Gbar}(\XIJ_z)$ and it
follows from~\autoref{cor:linind} that $\{\, t_\lambda z \mid z\in \WIJ,\
\lambda\in X(T)^+_{z} \,\}$ parametrizes an $A$-basis of $K^{\Gbar}(\XIJ)$.

In summary, there is a filtration of $K^{\Gbar}(\XIJ)$ such that the direct
summands of the associated graded $A$-module are naturally indexed by
$(W_I,W_J)$-double cosets in $W$. In addition, the summand indexed by $z$ in
$\WIJ$ is isomorphic to $K^{\Lbar_z}(\SN_z)$ and has a basis consisting of
isomorphism classes of equivariant coherent sheaves canonically indexed by
the $(W_I,W_J)$-double cosets in $\Wex$ that project to $W_IzW_J$.

\section{The proof of \autoref{thm:main}}\label{sec:proof}

In this section we prove~\autoref{thm:main}. The proof proceeds in three
steps. The first step is show that $K^{\Gbar}(\XIJ)$ is a free $A$-module,
the second step is to show that the composition $\etaIJ_* \varphi$ factors
through $\chi^{IJ}$, and the third step is to show that the resulting map
$\psi^{IJ}$ is an isomorphism. In the course of the argument we construct
explicit $A$-bases of of $\CHIJ$ and $K^{\Gbar}(\XIJ)$ that correspond under
$\psi^{IJ}$.

To show that $K^{\Gbar}(\XIJ)$ is a free $A$-module, we use the filtration
on $K^{\Gbar}(\XIJ)$ determined by the $G$-orbits on $G/P_I \times G/P_J$.
Recall that the rule $z\mapsto G\cdot (P_I,zP_J)$ defines a bijection
between $\WIJ$ and the set of $G$-orbits in $G/P_I \times G/P_J$.  Let
$q_{IJ}\colon \XIJ\to G/P_I \times G/P_J$ be the projection on the second
and third factors, and for $z$ in $\WIJ$ define $\XIJ_z$ to be the preimage
in $\XIJ$ of the orbit $G\cdot (P_I, zP_J)$. Then
\[
\XIJ_z=\{\, (x, gP_I, gzP_J)\in \FN\times G/P_I\times G/P_J \mid g\inverse
x\in \FN\cap \fp_I\cap z\fp_J\,\}.
\]
Choose a linear order on $\WIJ$, say $\WIJ=\{\, z_i\mid 1\leq i\leq
|\WIJ|\,\}$, that extends the Bruhat order and define
\[
\XIJ_{\ssleq i}=\coprod_{j\leq i} \XIJ_{z_j}.
\] 
Then $\XIJ_{\ssleq i}= \XIJ_{\ssleq i-1} \amalg \XIJ_{z_i}$, where
$\XIJ_{\ssleq i-1}$ is closed in $\XIJ_{\ssleq i}$ and $\XIJ_{{z_i}}$ is
open in $\XIJ_{\ssleq i}$. It is shown in \autoref{ssec:ex} that for
$i\geq1$ the sequence
\begin{equation}
  \label{eq:exz}
  \xymatrix{0 \ar[r] & K^{\Gbar}(\XIJ_{\ssleq i-1}) \ar[r]^-{()_*} &
    K^{\Gbar}(\XIJ_{\ssleq i}) \ar[r]^-{()^*} & K^{\Gbar}(\XIJ_{z_i}) \ar[r]&
    0 } 
\end{equation}
of $R(\Gbar)$-modules is exact. It follows that the embedding $\XIJ_{\ssleq
  i}\hookrightarrow \XIJ$ induces an injection $K^{\Gbar}(\XIJ_{\ssleq i})
\hookrightarrow K^{\Gbar}(\XIJ)$ in equivariant $K$-theory. 

Suppose $z$ is in $\WIJ$ and recall that $L_z=L_I\cap {}^zL_J$ and that
$\SN_z=\FN\cap \fl_z$ is the cone of nilpotent elements in $\fl_z$. It is
shown in \autoref{ssec:xijz}, in the course of the proof
of~\autoref{thm:ost}, that $K^{\Gbar}(\XIJ_z) \cong
K^{\Lbar_z}(\SN_z)$. Thus, the next proposition follows from the exact
sequence~\eqref{eq:exz} by induction and~\autoref{lem:ostlem}.

\begin{proposition}
  The equivariant $K$-group $K^{\Gbar}(\XIJ)$ is a free $A$-module.
\end{proposition}

Next, to show that $\etaIJ_* \varphi$ factors through $\chi^{IJ}$, it is
enough to show that for each generator $s$ in $J$, and each generator $t$ in
$I$,
\begin{equation}
  \label{eq:4b}
  \etaIJ_*  \varphi (hT_s)= - \etaIJ_* \varphi(h)= \etaIJ_*  \varphi (T_th)  
\end{equation}
for all $h$ in $\CH$. Indeed, if this condition holds, then for all
$h$ in $\CH$, $x$ in $W_I$, and $y$ in $W_J$, $\etaIJ_* \varphi(T_x h
T_y)= \epsilon_x \epsilon_y \etaIJ_*\varphi(h)$, and so
$\etaIJ_*\varphi(C_{w_I} h C_{w_J}) = r_{IJ} \etaIJ_*\varphi(h)$,
where $r_{IJ}$ is a non-zero element of $A$ that depends only on $I$
and $J$. Then $\psi^{IJ}\colon \CHIJ\to K^{\Gbar}(\XIJ)$ by
$\psi^{IJ}(C_{w_I} h C_{w_J}) = \etaIJ_*\varphi(h)$ is well defined
because if $C_{w_I} h C_{w_J}= 0$, then $r_{IJ} \etaIJ_*\varphi(h)=0$
and so $\etaIJ_*\varphi(h)=0$.

Let $\pi_I\colon \XEJ \to \XIJ$ by $\pi_I(x, gB,
hP_J)= (x, gP_I, hP_J)$ and let $\pi_J\colon \XIE \to \XIJ$ by $\pi_J(x,
gP_I, hB)= (x, gP_I, hP_J)$. Then the diagram
\begin{equation*}
  \label{eq:3b}
  \xymatrix{Z\ar[r]^{\eta^{\emptyset J}} \ar[d]_{\eta^{I \emptyset}}
        \ar[dr]^{\etaIJ} &  \XEJ \ar[d]^{\pi_I} \\ 
    \XIE \ar[r]_{\pi_J} &\XIJ}
\end{equation*}
commutes and so~\eqref{eq:4b} follows from the equalities
\[
\eta^{\emptyset J}_* \varphi (hT_s)= - \eta^{\emptyset J}_*
\varphi(h)\quad\text{and}\quad \eta^{I \emptyset}_* \varphi (T_th)= -
\eta^{I \emptyset}_* \varphi(h),
\]
for $s\in J$ and $t\in I$, by applying $(\pi_I)_*$ and $(\pi_J)_*$,
respectively. Thus, by symmetry, it is enough to show that $\eta^{\emptyset
  J}_* \varphi$ factors through the projection $\CH\to \CH C_{w_J}$ given by
right multiplication by $C_{w_J}$:
\[
\xymatrix{ \CH \ar@{-->}[d] \ar[r]^{\varphi} & K^{\Gbar}(Z)
  \ar[d]^-{\eta^{\emptyset J}_*} \\
  \CH C_{w_J} \ar@{-->}[r] & K^{\Gbar}(\XEJ).} 
\]

The intersection/Tor-product construction described by Lusztig in
\cite[\S6.4]{lusztig:bases} can be used to define a
$K^{\Gbar}(Z)$-module structure on $K^{\Gbar}(\XEJ)$, say
\[
\star_J\colon K^{\Gbar}(Z) \times K^{\Gbar}(\XEJ) \to K^{\Gbar}(\XEJ),
\]
such that the map $\eta^{\emptyset J}_*\colon K^{\Gbar}(Z) \to
K^{\Gbar}(\XEJ)$ is $K^{\Gbar}(Z)$-linear (see~\autoref{pro:xlin}). Thus,
for all $h$ in $\CH$ and all $s$ in $S$,
\[
\eta^{\emptyset J}_* \varphi(hT_s)= \eta^{\emptyset J}_*(\varphi(h) \star
\varphi(T_s)) = \varphi(h) \star_J \eta^{\emptyset J}_* (\varphi(T_s)),
\]
where $\star$ is the convolution product in $K^{\Gbar}(Z)$. Hence, it is
enough to show that $\eta^{\emptyset J}_* (\varphi(T_s))= -\eta^{\emptyset
  J}_*( \varphi(1))$ for all $s$ in $J$, or equivalently that
$\eta^{\emptyset J}_*(\varphi(T_s+1))=0$. For a simple reflection $s$ in
$W$, let $\mathbf a_s$ in $K^{\Gbar}(Z)$ be defined as in
\cite[\S7.20]{lusztig:bases} (see \autoref{sec:as}). By~\cite[\S7.25]
{lusztig:bases}, $\varphi(T_s+1)=-v \mathbf a_s$. Therefore, the existence
of the map $\psi^{IJ}$ in~\autoref{thm:main} follows from the next theorem.

\begin{theorem}\label{thm:as}
  Suppose $s$ is in $J$. Then $\etaEJ_*(\mathbf a_s)= 0$ in
  $K^{\Gbar}(\XEJ)$.
\end{theorem}

This theorem is proved in \autoref{sec:as}.

Finally, to show that $\psi^{IJ}$ is an isomorphism, we use an $A$-basis of
$\CHIJ$ and the partition $\XIJ= \coprod_{z\in \WIJ} \XIJ_z$ to define
compatible filtrations on $\CHIJ$ and $K^{\Gbar}(\XIJ)$. We also need the
analogous constructions for $\CH$ and $K^{\Gbar}(Z)$. 

Recall that $q_{Z}\colon Z\to G/B \times G/B$ is the projection on the
second and third factors. For $w$ in $W$ define $Z_w$ to be the preimage in
$Z$ of the orbit $G\cdot (B, wB)$. Then
\[
Z_w= \{\, (x, gB, gwB)\in \FN\times G/B\times G/B \mid g\inverse x\in
\fu\cap w\fu\,\}.
\]
Define 
\[
Z_{\ssleq w}=\coprod_{y\leq w} Z_y \quad \text{and} \quad Z_{\sslt w}=
Z_{\ssleq w}\setminus Z_w= \coprod_{y< w} Z_y,
\]
where $\leq$ is the Bruhat order on $W$. Similarly, for $z$ in $\WIJ$ define
\[
\XIJ_{\ssleq z} =\coprod_{y\leq z} \XIJ_y \quad \text{and} \quad \XIJ_{\sslt
  z}= \XIJ_{\ssleq z}\setminus \XIJ_z =\coprod_{y< z} \XIJ_y,
\] 
where the unions are over $y$ in $\WIJ$.

Suppose $z$ is in $\WIJ$ and let $\eta^z$, $\eta^{\ssleq z}$, and
$\eta^{\sslt z}$ be the restrictions of $\etaIJ$ to $Z_z$, $Z_{\ssleq z}$,
and $Z_{\sslt z}$, respectively.  It is shown in
\cite[\S3]{douglassroehrle:geometry} that if $w_1$ is in $W_I$ and $w_2$ is
in $W_J$, then $\etaIJ(Z_{w_1zw_2})\subseteq \XIJ_z$, and that
$\etaIJ(Z_{w_1zw_2}) =\XIJ_z$ if and only if $w_1zw_2=z$. It is shown in
~\cite[Lemma 2.2] {douglass:inversion} that if $z$ and $z'$ are in $\WIJ$,
then $z\leq z'$ if and only if there are elements $w$ in $W_IzW_J$ and $w'$
in $W_Iz'W_J$ such that $w\leq w'$. Therefore, $\etaIJ(Z_{\ssleq
  z})\subseteq \XIJ_{\ssleq z}$, and letting $r_z$ and $r_z^{IJ}$ denote the
inclusions $Z_z\to Z_{\ssleq z}$ and $\XIJ_z\to \XIJ_{\ssleq z}$,
respectively, the square
\begin{equation}
  \label{eq:cart}
  \vcenter{\vbox{
      \xymatrix{Z_z\ar[r]^{r_z} \ar[d]^{\eta^z} & Z_{\ssleq
          z}\ar[d]^{\eta^{\ssleq z}} \\
        \XIJ_z\ar[r]^{r_z^{IJ}} & \XIJ_{\ssleq z}}
    }}  
\end{equation}
is cartesian. 

\begin{lemma}\label{lem:hyp1}
  The map $\eta^z\colon Z_z\to \XIJ_z$ is a proper morphism.
\end{lemma}

\begin{proof}
  Define $Z_{W_IzW_J} = \coprod_{w\in {W_IzW_J}} Z_w$. Then $Z_{{W_IzW_J}}=
  (\etaIJ)\inverse \big( \XIJ_z \big)$ and so by base change, the
  restriction of $\etaIJ$ to $Z_{{W_IzW_J}}$ is proper. Now $G\cdot (B,zB)$
  is closed in $\coprod_{w\in W_IzW_J} G\cdot (B, wB)$ and hence $Z_z$ is
  closed in $Z_{{W_IzW_J}}$. It follows that $\eta^z$ is proper.  \qed
\end{proof}

Since $\etaIJ(Z_{\leq z}) \subseteq \XIJ_{\leq z}$ for all $z$ in $\WIJ$
and~\eqref{eq:cart} is cartesian, the proper morphisms $\etaIJ$ and $\eta^z$
induce a map of short exact sequences such that the following diagram
commutes:
\begin{equation}
  \label{eq:exwz}
  \vcenter{\vbox{
      \xymatrix{0 \ar[r] & K^{\Gbar}(Z_{\sslt z}) \ar[r]^{()_*}
        \ar[d]^{\eta^{\sslt z}_*} &  K^{\Gbar}(Z_{\ssleq z}) \ar[r]^{()^*}
        \ar[d]^{\eta^{\ssleq  z}_*}& K^{\Gbar}(Z_{z})  \ar[r] \ar[d]^{\eta^{
            z}_*}& 0 \\ 
        0 \ar[r] & K^{\Gbar}(\XIJ_{\sslt z}) \ar[r]^{()_*} &
        K^{\Gbar}(\XIJ_{\ssleq z}) \ar[r]^{()^*} & K^{\Gbar}(\XIJ_{z})
        \ar[r]& 0 ,}   
    }}
\end{equation}
where $Z_{\sslt 1}= \XIJ_{\sslt 1}= \emptyset$.

Suppose $z$ is in $\WIJ$. Then $L_z = L_I\cap {}^zL_J =L_{\IzJ}$ is a Levi
subgroup of $P_I\cap {}^zP_J$ (see~\cite[\S69B] {curtisreiner:methodsII}).
Set
\[
B_z=L_z\cap B \quad\text{and} \quad \SNt_z= \{\, (x, hB_z)\in \SN_z\times
L_z/B_z\mid h\inverse x\in \fb_z\,\},
\]
and let 
\[
p_z\colon \SNt_z\to \SN_z\quad\text{and}\quad q_z\colon \SNt_z\to L_z/B_z
\]
be the projections. Then $B_z$ is a Borel subgroup of $L_z$ and $p_z$ is the
Springer resolution of $\SN_z$.

\begin{theorem}\label{thm:ost}
  Suppose $z$ is in $\WIJ$. Then there is a commutative diagram of
  $R(\Gbar)$-modules
  \begin{equation*}
    \xymatrix{ K^{\Gbar}(Z_z) \ar[r]_{\cong} \ar[d]^{\eta^z_*} &
      K^{\Lbar_z}(\SNt_z) \ar[d]^{(p_z)_*} \\ 
      K^{\Gbar}(\XIJ_z) \ar[r]_{\cong} & K^{\Lbar_z}(\SN_z), }  
  \end{equation*}
  where the horizontal maps are isomorphisms and the vertical maps are
  surjections.
\end{theorem}

This theorem is proved in \autoref{ssec:xijz}.

\begin{corollary}\label{cor:surj}
  The map $\eta^{\ssleq z}_*\colon K^{\Gbar}(Z_{\ssleq z}) \to
  K^{\Gbar}(\XIJ_{\ssleq z})$ is surjective for all $z$ in
  $\WIJ$. Therefore, the maps $\etaIJ_*\colon K^{\Gbar}(Z) \to
  K^{\Gbar}(\XIJ)$ and $\psi^{IJ}\colon \CHIJ\to K^{\Gbar}(\XIJ)$ are
  both surjective.
\end{corollary}

\begin{proof}
  We show that $\eta^{\ssleq z}_*$ is surjective using induction on $z$ in
  the Bruhat order on $\WIJ$. It is shown
  in~\cite[\S3.17]{kazhdanlusztig:langlands} that the image of
  $K^{\Gbar}(Z_{\sslt z})$ in $K^{\Gbar}(Z)$ coincides with the sum of the
  images of the spaces $K^{\Gbar}(Z_{\ssleq y})$ for $y<w$. A similar
  argument, using~\autoref{lem:k1}, shows that the image of
  $K^{\Gbar}(\XIJ_{\sslt z})$ in $K^{\Gbar}(\XIJ)$ coincides with the sum of
  the images of the spaces $K^{\Gbar}(\XIJ_{\ssleq y})$ for $y$ in $\WIJ$
  with $y<z$. Thus, the natural maps
  \[
  \bigoplus_{y< z} K^{\Gbar}(Z_{\ssleq y}) \xrightarrow{\ f_1\ }
  K^{\Gbar}(Z_{\sslt z}) \quad\text{and} \quad \bigoplus_{\substack{y\in
      \WIJ \\ y< z}} K^{\Gbar}(\XIJ_{\ssleq y}) \xrightarrow{\ f_2\ }
  K^{\Gbar}(\XIJ_{\sslt z})
  \]
  are surjective. The composition $\bigoplus_{y< z} K^{\Gbar}(Z_{\ssleq y})
  \xrightarrow{\ f_1\ } K^{\Gbar}(Z_{\sslt z}) \xrightarrow{\eta^{\sslt
      z}_*} K^{\Gbar}(\XIJ_{\sslt z})$ is equal the composition
  $\bigoplus_{y< z} K^{\Gbar}(Z_{\ssleq y}) \xrightarrow{\ \ }
  \bigoplus_{\substack{y\in \WIJ \\ y< z}} K^{\Gbar}(\XIJ_{\ssleq y})
  \xrightarrow{\ f_2\ } K^{\Gbar}(\XIJ_{\sslt z})$. By induction, and the
  fact that the map $f_2$ is surjective, this last composition is
  surjective. It follows that $\eta^{\sslt z}_*$ is surjective. Now the fact
  that $\eta^{\ssleq z}_*$ is surjective follows from
  diagram~\eqref{eq:exwz} by the Five Lemma and~\autoref{thm:ost}.

  Let $z_0$ be the element in $\WIJ$ with maximal length. Then $\XIJ_{\ssleq
    z_0}= \XIJ$ and so taking $z=z_0$ we see that the map
  $K^{\Gbar}(Z_{\ssleq z_0})\to K^{\Gbar}(\XIJ)$ induced by $\etaIJ$ is a
  surjection. This map is $\etaIJ_* \circ (j_{z_0})_*$, where $j_{z_0}\colon
  Z_{\ssleq z_0} \to Z$ is the inclusion. Therefore, $\etaIJ_*$ is
  surjective and so $\psi^{IJ}$ is surjective. \qed
\end{proof}

To define filtrations on $\CH$ and $\CHIJ$ compatible with
diagram~\eqref{eq:exwz} we need to study $(W_I,W_J)$-double cosets in
$\Wex$. For $I\subseteq S$ define
\[
X(T)_I^+ =\{\, \lambda\in X(T) \mid \langle \lambda, \check \alpha \rangle
\geq0\ \forall \alpha\in \Pi_I\,\},
\]
where $\langle \,\cdot\,, \,\cdot\,\rangle$ is the usual pairing between
$X(T)$ and $\Hom(\BBC^*, T)$, the group of co-characters of $T$, and $\{\,
\check \alpha\mid \alpha\in \Phi\,\}$ is the dual root system in
$\Hom(\BBC^*, T)$. Then $X(T)_{\emptyset}^+=X(T)$ and $X(T)_S^+=X(T)^+$, the
set of dominant weights.  Because $X(T)^+$ is a fundamental domain for the
action of $W$ on $X(T)$, it follows that $\{\, t_\lambda\mid \lambda\in
X(T)^+\,\}$ is a set of $(W,W)$-double coset representatives in $\Wex$. The
next lemma is the analog of this fact for $(W_I,W_J)$-double cosets.

\begin{lemma}\label{lem:double}
  Every $(W_I,W_J)$-double coset in $\Wex$ contains a unique representative
  of the form $t_\lambda z$, where $z$ is in $\WIJ$ and $\lambda$ is in
  $X(T)_{\IzJ}^+$.
\end{lemma}

\begin{proof}
  If $\lambda$ is in $X(T)$, $w$ is in $W$, $w_1$ is in $W_I$, and
  $w_2$ is in $W_J$, then $w_1t_\lambda ww_2= t_{w_1(\lambda)}
  w_1ww_2$. Thus, every $(W_I,W_J)$-double coset has a representative
  of the form $t_\lambda z$ where $z$ is in $\WIJ$. Suppose that $z$
  is in $\WIJ$ and $\lambda$ and $\mu$ are in $X(T)$. Then $t_\mu z=
  w_1t_\lambda zw_2$ for some $w_1$ in $W_I$ and $w_2$ in $W_J$ if and
  only if $\mu= w_1(\lambda)$ and $z=w_1zw_2$. Therefore, $t_\lambda
  z$ and $t_\mu z$ are in the same $(W_I,W_J)$-double coset if and
  only if $\mu= w_1(\lambda)$ for some $w_1$ in $W_I\cap
  {}^zW_J$. Because $W_I\cap {}^zW_J= W_{\IzJ}$ (see
  \cite[\S64C]{curtisreiner:methodsII}), we see that $t_\lambda z$ and
  $t_\mu z$ are in the same $(W_I,W_J)$-double coset if and only if
  $\mu= w_1(\lambda)$ for some $w_1$ in $W_{\IzJ}$. Finally,
  $X(T)_{\IzJ}^+$ is a fundamental domain for the action of $W_{\IzJ}$
  on $X(T)$ and so every $(W_I,W_J)$-double coset in $\Wex$ contains a
  unique representative of the form $t_\lambda z$, where $z$ is in
  $\WIJ$ and $\lambda$ is in $X(T)_{\IzJ}^+$. \qed
\end{proof}

For $z$ in $\WIJ$, set 
\[
X(T)_z^+=X(T)_{\IzJ}^+.
\]
Then $X(T)_z^+$ is the set of weights that are dominant for $L_{z}$ with
respect to the Borel subgroup $B_z$.

Let $\epsilon_I=\sum_{w\in W_I} \epsilon_ww$ for $I\subseteq
S$. Then~\autoref{lem:double} shows that $\{\, \epsilon_I t_\lambda z
\epsilon_J \mid z\in \WIJ,\, \lambda \in X(T)_{z} ^+ \,\}$ is a $\BBZ$-basis
of $\epsilon_I \BBZ[\Wex] \epsilon_J$. A similar statement is true for
$\CHIJ$.

\begin{theorem}\label{thm:hijbasis}
  The set $\{\, C_{w_I} \theta_\lambda T_z C_{w_J} \mid z\in \WIJ,\, \lambda
  \in X(T)_{z} ^+ \,\}$ is an $A$-basis of $\CHIJ$.
\end{theorem}

This theorem is proved in \autoref{sec:basis}.

Recall that $\{\, \theta_\lambda T_w\mid \lambda\in X(T), w\in W\,\}$ is the
Bernstein basis of $\CH$. By analogy, we call $\{\, C_{w_I} \theta_\lambda
T_z C_{w_J} \mid z\in \WIJ,\, \lambda \in X(T)_{z} ^+ \,\}$ the
\emph{Bernstein basis} of $\CHIJ$. We will see that
\[
\{\, \psi^{IJ}(C_{w_I} \theta_\lambda T_z C_{w_J})\mid z\in \WIJ,\, \lambda
\in X(T)_{z} ^+ \,\}
\]
is an $A$-basis of $K^{\Gbar}(\XIJ)$.

For $w$ in $W$ and $z$ in $\WIJ$ define
\begin{align*}
  \CH_{w}&= \operatorname{span}_A\{\, \theta_\lambda T_w\mid \lambda \in
  X(T)\,\}, & && \CH_{\ssleq w}&= \sum_{y\leq w} \CH_y, \\
  \CHIJ_{z}&= \operatorname{span}_A\{\, C_{w_I} \theta_\lambda T_zC_{w_J}
  \mid \lambda \in X(T)\,\}, & &\text{and}& \CHIJ_{\ssleq z}&= \sum_{y\leq
    z} \CHIJ_{y},
\end{align*}
where the last sum is over $y$ in $\WIJ$. The proof
of~\autoref{thm:hijbasis} shows that in fact $\{\, C_{w_I} \theta_\lambda
T_zC_{w_J} \mid \lambda \in X(T)^+_z\,\}$ is an $A$-basis of
$\CHIJ_z$. Extending the definition of $\CL_\lambda$ for $\lambda$ in
$X(T)$, let $\CL^z_\lambda$ be the canonical $\Lbar_z$-equivariant line
bundle on $L_z/B_z$ determined by $\lambda$. Also, let $\chi^{\ssleq
  z}\colon \CH_{\ssleq z}\to \CHIJ_{\ssleq z}$ be the restriction of
$\chi^{IJ}$. We can now extend the diagram in~\autoref{thm:ost} to include
$\CH_{\ssleq z}$ and $\CHIJ_{\ssleq z}$.

\begin{theorem}\label{thm:wideiso2}
  Suppose $z$ is in $\WIJ$. Then there is a commutative diagram of
  $A$-modules
  \begin{equation} 
    \label{eq:wideiso}
    \vcenter{\vbox{    
        \xymatrix{%
          \CH_{\ssleq z} \ar[r]^-{}_-{\cong} \ar[d]^{\chi^{\ssleq z}}& 
          K^{\Gbar}(Z_{\ssleq z}) \ar@{->>}[r]^-{r_z^*} \ar[d]^{\eta^{\ssleq
              z}_*}& K^{\Gbar}(Z_z) \ar[r]_-{\cong} \ar[d]^{\eta^z_*} &
          K^{\Lbar_z}(\SNt_z) \ar[d]^{(p_z)_*} \\ 
          \CHIJ_{\ssleq z} \ar[r]^-{}_-{\cong} & K^{\Gbar}(\XIJ_{\ssleq z})
          \ar@{->>}[r]^-{(r_z^{IJ})^*} & K^{\Gbar}(\XIJ_z) \ar[r]_-{\cong} &
          K^{\Lbar_z}(\SN_z), }
      }}   
  \end{equation}
  where
  \begin{enumerate}
  \item the horizontal maps are surjective or bijective, as indicated, and
    the vertical maps are surjective, \label{i:wi1}
  \item if $f_1$ is the composition across the top row, then $\CH_y$ is in
    the kernel of $f_1$ for all $y<z$ and the restriction of $f_1$ to
    $\CH_{z}$ is the isomorphism that carries $\theta_\lambda T_z$ to
    $\epsilon_z q_z^*[\CL^z_\lambda]$ for $\lambda$ in $X(T)$,
    and \label{i:wi2}
  \item if $f_2$ is the composition across the bottom row, then $\CHIJ_y$ is
    in the kernel of $f_2$ for all $y<z$ in $\WIJ$ and the restriction of
    $f_2$ to $\CHIJ_{z}$ is the isomorphism that carries $C_{w_I}
    \theta_\lambda T_zC_{w_J}$ to $\epsilon_z (p_z)_*q_z^*[\CL^z_\lambda]$
    for $\lambda$ in $X(T)_{z}^+$. \label{i:wi3}
  \end{enumerate}
\end{theorem}

This theorem is proved in~\autoref{ssec:isoz1}\,--\,\autoref{ssec:isoz3}.

The isomorphisms in the left-hand square in~\eqref{eq:wideiso} may be viewed
as the restrictions of $\varphi$ and $\psi^{IJ}$ to $\CH_{\ssleq z}$ and
$\CHIJ_{\ssleq z}$, respectively in the sense that the inclusions
$\CH_{\ssleq z}\subseteq \CH$ and $\CHIJ_{\ssleq z} \subseteq \CHIJ$, and
the closed embeddings $Z_{\ssleq z} \subseteq Z$ and $\XIJ_{\ssleq z}
\subseteq \XIJ$, induce commutative diagrams
\[
\vcenter{\vbox{ \xymatrix{ \CH_{\ssleq z} \ar[r]^-{}_-{\cong} \ar[d]&
      K^{\Gbar}(Z_{\ssleq z}) \ar[d]\\
      \CH \ar[r]^-{\varphi} & K^{\Gbar}(Z)} }} \quad\text{and} \quad
\vcenter{\vbox{ \xymatrix{\CHIJ_{\ssleq z} \ar[r]^-{}_-{\cong} \ar[d] &
      K^{\Gbar}(\XIJ_{\ssleq z}) \ar[d] \\
      \CHIJ \ar[r]^-{\psi^{IJ}} & K^{\Gbar}(\XIJ) } }}
\]
where the vertical maps are injective.

Let $z_0$ be the element in $\WIJ$ with maximal length. Then $\CHIJ_{\ssleq
  z_0} = \CHIJ$ and $\XIJ_{\ssleq z_0}= \XIJ$. Thus, taking $z=z_0$ in the
theorem we have that $\psi^{IJ} \colon \CHIJ\to K^{\Gbar}(\XIJ)$ is an
isomorphism. This completes the proof of~\autoref{thm:main}.

\begin{corollary}\label{cor:linind}
  The set $\{\, \psi^{IJ}(C_{w_I} \theta_\lambda T_z C_{w_J})\mid z\in
  \WIJ,\, \lambda \in X(T)_{z} ^+ \,\}$ is an $A$-basis of
  $K^{\Gbar}(\XIJ)$.
\end{corollary}

It follows from~\autoref{thm:wideiso2} that $\psi^{IJ}(\CHIJ_{\ssleq z}) =
K^{\Gbar}(\XIJ_{\ssleq z})$ for $z$ in $\WIJ$, and so the following
refinement of~\autoref{thm:main} is a consequence of the constructions
above.

\begin{corollary}\label{cor:ref}
  The $A$-submodules $K^{\Gbar}(\XIJ_{\ssleq z})$ determine a filtration of
  $K^{\Gbar}(\XIJ)$ such that if $\gr\, K^{\Gbar}(\XIJ)$ is the associated
  graded $A$-module, then there are $A$-module isomorphisms
  \[
  \gr\, K^{\Gbar}(\XIJ) \cong \bigoplus_{z\in \WIJ} \CHIJ_z \cong
  \bigoplus_{z\in \WIJ} K^{\Lbar_z} \left(\SN_{z} \right).
  \]
\end{corollary}

\section{The ``bar'' involution and Kazhdan-Lusztig basis of
  \texorpdfstring{$\CHIJ$}{}} \label{sec:kl}

For $z$ in $\WIJ$ and $\lambda$ in $X(T)_z^+$, let $m_{\lambda,z}$ be the
element in the double coset $W_It_\lambda zW_J$ with minimal length and
define
\[
T_{\lambda,z}=\chi^{IJ}( T_{m_{\lambda,z}} )=C_{w_I} T_{m_{\lambda,z}}
C_{w_J}
\]
in $\CHIJ$.

\begin{lemma}\label{lem:sb}
  The set $\{\, T_{\lambda,z}\mid z\in \WIJ,\ \lambda\in X(T)_z^+ \,\}$ is
  an $A$-basis of $\CHIJ$.
\end{lemma}

\begin{proof}
  If $x$ is in $W_It_\lambda zW_J$, then there are elements $w_1$ in $W_I$
  and $w_2$ in $W_J$ so that $x=w_1m_{\lambda,z} w_2$ and $\ell(x)=
  \ell(w_1)+ \ell(m_{\lambda,z}) +\ell(w_2)$. Thus, it follows
  from~\eqref{eq:cxts} that $C_{w_I} T_xC_{w_J}= \pm C_{w_I}
  T_{m_{\lambda,z}} C_{w_J}$. Therefore $\{\, T_{\lambda,z}\mid z\in \WIJ,\
  \lambda\in X(T)_z^+ \,\}$ spans $\CHIJ$. By~\autoref{lem:double} $\{\,
  \epsilon_I m_{\lambda,z} \epsilon_J \mid z\in \WIJ,\ \lambda\in X(T)_z^+
  \,\}$ is a $\BBZ$-basis of $\epsilon_I \BBZ[\Wex] \epsilon_J$. It follows
  that $\{\, T_{ \lambda,z}\mid z\in \WIJ,\ \lambda\in X(T)_z^+ \,\}$ is
  linearly independent. \qed
\end{proof}

We call $\{\, T_{\lambda,z}\mid z\in \WIJ,\ \lambda\in X(T)_z^+ \,\}$ the
\emph{standard basis} of $\CHIJ$. It follows from \eqref{eq:cx} that the
ring involution $\overline{\phantom {x} }\colon \CH \to \CH$ induces an
involution of abelian groups, also denoted by $\overline{\phantom {x} }$,
from $\CHIJ$ to $\CHIJ$ that we call the \emph{bar involution} of $\CHIJ$.

Now consider the Kazhdan-Lusztig basis $\{\, C_x'\mid x\in \Wex\,\}$ of
$\CH$. If $t$ is in $I$ and $x$ is in $\Wex$ with $tx<x$, then
by~\eqref{eq:cxts}
\[
C_{w_I} C_x'C_{w_J}= -C_{w_I}T_t C_x'C_{w_J}= -v^2 C_{w_I} C_x'C_{w_J},
\]
and so $C_{w_I} C_x'C_{w_J}=0$. Similarly, if $s$ is in $J$ and $x$ is in
$\Wex$ with $xs<x$, then $C_{w_I} C_x'C_{w_J}=0$. It follows that $C_{w_I}
C_x'C_{w_J}=0$ unless $x=m_{\lambda,z}$ for some $\lambda, z$. For $z$ in
$\WIJ$ and $\lambda$ in $X(T)_z^+$, define
\[
C_{\lambda,z}'=\chi^{IJ}(C_{m_{\lambda,z}}') =C_{w_I} C_{m_{\lambda,z}}'
C_{w_J}
\]
in $\CHIJ$. As usual, an element $h$ in $\CH$ or $\CHIJ$ is said to be
\emph{self-dual} if $\overline{ h} =h$.

\begin{proposition}
  The set $\{\, C_{\lambda, z}'\mid z\in \WIJ,\ \lambda\in X(T)_z^+\,\}$ is
  a self-dual basis of $\CHIJ$ and the map $\chi^{IJ}\colon \CH\to \CHIJ$
  may be identified with projection of $\CH$ onto the span of $\{\,
  C_{m_{\lambda, z}}'\mid z\in \WIJ,\ \lambda\in X(T)_z^+\,\}$, that is,
  \[
  \chi^{IJ}(C_x')=\begin{cases} C_{\lambda,z}'& \text{if $x=m_{\lambda,z}$
      for some $z\in \WIJ$, $\lambda\in X(T)_z^+$,} \\ 0&\text{otherwise.}
  \end{cases}
  \]
\end{proposition}

\begin{proof}
  We need to show that $\{\, C_{\lambda, z}'\mid z\in \WIJ,\ \lambda\in
  X(T)_z^+\,\}$ is a basis of $\CHIJ$. Suppose $x$ is in $\Wex$, then
  applying $\chi^{IJ}$ to both sides of the formula for $C_x'$ given
  in~\eqref{eq:cx'} we obtain
  \[
  C_{w_I}C_x' C_{w_J}=\sum_{y\leq x} v_x^{-1} P_{y,x} C_{w_I}T_y C_{w_J}.
  \]
  If the left-hand side is non-zero, then it is $\chi^{IJ}( C_{m_{
      \lambda,z}}') = C_{\lambda,z}'$ for some $z$ and $\lambda$. We've
  observed that $C_{w_I}T_y C_{w_J}= \pm C_{w_I} T_{\lambda', z'} C_{w_J}$
  when $y$ is in $W_Im_{\lambda', z'}W_J$. Moreover, it follows from the
  definition of the Bruhat order on $\Wex$ and~\cite[Lemma 2.2]
  {douglass:inversion} that if $y<x$ and $y$ is in $W_Im_{\lambda', z'}W_J$,
  then $m_{\lambda', z'}< m_{\lambda, z}$. Therefore,
  \begin{equation}
    \label{eq:mob}
    C_{\lambda, z}'= v_{m_{\lambda,z}}\inverse T_{\lambda,z} +
    \sum_{m_{\lambda',z'}< m_{\lambda, z} } p_{\lambda', z'} T_{\lambda', z'}
  \end{equation}
  for some polynomials $p_{\lambda', z'}$ in $A$. It follows that the system
  of equations~\eqref{eq:mob} indexed by pairs $(z, \lambda)$ can be solved
  for $\{\, T_{\lambda,z} \mid z\in\WIJ,\ \lambda\in X(T)^+_z\,\}$ and so
  $\{\, C_{\lambda, z}'\mid z\in \WIJ,\ \lambda\in X(T)_z^+\,\}$ is a basis
  of $\CHIJ$. \qed
\end{proof}

We call $\{\, C_{\lambda, z}'\mid z\in \WIJ,\ \lambda\in X(T)_z^+\,\}$ the
\emph{Kazhdan-Lusztig basis} of $\CHIJ$. The elements of the Kazhdan-Lusztig
basis have the standard uniqueness property, analogous to~\eqref{eq:cx'}.
The proof of~\cite[Theorem 1.1] {kazhdanlusztig:coxeter} is easily adapted
to prove the following (see~\cite[Claim 2.4]{soergel:kazhdan}).

\begin{proposition}
  For $z$ in $\WIJ$ and $\lambda$ in $X(T)_z^+$, $C_{\lambda,z}'$ is the
  unique element in $\CHIJ$ such that
  \begin{equation*}
    \label{eq:cxij'}
    \begin{cases}
      \overline{C_{\lambda,z}'}= C_{\lambda,z}'\quad \text{and} \\
      C_{\lambda,z}'=v_{m_{\lambda,z}}\inverse T_{\lambda,z}+
      \sum_{m_{\lambda',z'} < m_{\lambda,z}} v_{m_{\lambda,z}} \inverse
      P_{(\lambda', z'),(\lambda,z)}T_{\lambda',z'},
    \end{cases}
  \end{equation*}
  where $P_{(\lambda', z'),(\lambda,z)}$ is a polynomial in $v$ of degree
  at most $\ell(m_{\lambda,z})-\ell(m_{\lambda',z'}) -1$.
\end{proposition}

To summarize, we have the standard basis 
\[
\{\, T_{\lambda,z}\mid z\in \WIJ,\ \lambda\in X(T)_z^+ \,\} = \{\,
\chi^{IJ}(T_{m_{\lambda,z}})\mid z\in \WIJ,\ \lambda\in X(T)_z^+ \,\},
\]
the Bernstein basis
\[
\{\, C_{w_I} \theta_\lambda T_z C_{w_J}\mid z\in \WIJ,\, \lambda \in
X(T)_{z} ^+ \,\} = \{\, \chi^{IJ}(\theta_\lambda T_z)\mid z\in \WIJ,\,
\lambda \in X(T)_{z} ^+ \,\},
\]
and the Kazhdan-Lusztig basis
\[
\{\, C_{\lambda, z}'\mid z\in \WIJ,\ \lambda\in X(T)_z^+\,\} = \{\,
\chi^{IJ}(C_{m_{\lambda, z}}') \mid z\in \WIJ,\ \lambda\in X(T)_z^+\,\}
\]
of $\CHIJ$. Applying $\psi^{IJ}$ to each of these bases we obtain standard,
Bernstein, and Kazhdan-Lusztig bases of $K^{\Gbar}(\XIJ)$.

\section{Intersection/Tor products} \label{sec:kthy}

In this section we formalize the intersection/Tor product
constructions described in~\cite[\S6.7]{lusztig:bases} in the form
they are used in this paper.

\subsection{The basic construction} 

Suppose that $V$, $A$, $B$, and $C$ are $H$-varieties such that $V$ is
smooth. Let $V_1$ and $V_2$ be closed, $H$-stable subvarieties of $V$
and let $p_{1}\colon V_1\to A$, $p_{2}\colon V_2\to B$, and
$p_{12}\colon V_1\cap V_2\to C$ be $H$-equivariant maps such that
$p_1$ and $p_2$ are smooth and $p_{12}$ is proper. These spaces and
maps can be assembled in the commutative diagram
\begin{equation}
  \label{eq:tor} 
  \vcenter{\vbox{
      {\tiny
        \xymatrix{ C&& B \\
          & V_1\cap V_2\ar[ul]_{p_{12}} \ar[r]\ar[d] &V_2 \ar[u]^{p_{2}}
          \ar[d]  \\ 
          A  &V_1 \ar[l]_{p_{1}} \ar[r]&V,} 
      } 
    }}
\end{equation}
where the unlabeled maps are inclusions.

The intersection/Tor-product construction described by Lusztig
\cite[\S6.4]{lusztig:bases} can be used to define an $R(H)$-bilinear map
$\star\colon K^{H}(A) \times K^H(B) \to K^H(C)$ as in
\cite[\S7]{lusztig:bases}. Precisely, for $\CF$ and $\CG$ in $K^H(A)$ and
$K^H(B)$, respectively, define
\begin{equation*}
  \label{eq:star} 
  \CF \star \CG = (p_{12})_* \left(p_{1}^*\,\CF \otimes_{V}^L p_{2}^*\,
    \CG\right) \in K^{H}(C) ,
\end{equation*}
where the Tor-product $\otimes_{V}^L$ is with respect to the smooth
variety $V$ and its closed subvarieties $V_1$ and $V_2$.

Next suppose that $V'$, $A'$, $B'$, $C'$, $V_1'$, $V_2'$, $p_{1}'$,
$p_{2}'$, and $p_{12}'$ are as in~\eqref{eq:tor}. Then we have a commutative
diagram analogous to~\eqref{eq:tor} that we connote
by~$\text{\eqref{eq:tor}}'$, and an $R(H)$-bilinear map $\star'\colon
K^{H}(A') \times K^H(B') \to K^H(C')$.

By a \emph{morphism} from~\eqref{eq:tor} to~$\text{\eqref{eq:tor}}'$ we mean a
tuple $(\mu,f,g,h)$ of $H$-equivariant morphisms where $\mu\colon V\to V'$,
$f\colon A\to A'$, $g\colon B\to B'$, and $h\colon C\to C'$ satisfy the
following conditions.
\begin{itemize}
\item $\mu$ is smooth, $V_1\subseteq \mu\inverse(V_1')$, and $V_2\subseteq
  \mu\inverse(V_2')$.

  \noindent Let $\mu_1$, $\mu_2$, and $\mu_{12}$ denote the restrictions of
  $\mu$ to maps $V_1\to V_1'$, $V_2\to V_2'$, and $V_1\cap V_2\to V_1' \cap
  V_2'$, respectively.

\item $fp_1= p_1'\mu_1$, $gp_2= p_2'\mu_2$, and $hp_{12}=
  p_{12}'\mu_{12}$. 
\end{itemize}
Then the following diagram commutes:
\begin{equation}\label{eq:dmap} {\tiny
\xymatrix@1@!@R=0em@C=0em{%
&&&& B \ar[dr]^{g}&\\
C\ar[dr]_{h}&&&&&B'\\
&C'&V_1\cap V_2 \ar[ull]_{p_{12}} \ar[rr] \ar[dr]^{\mu_{12}}
\ar[dd]|(.25)\hole &&V_2 \ar[uu]^{p_{2}} \ar[dr]_{}^{\mu_2}
\ar'[d][dd]^{}&\\ 
&&&V_1'\cap V_2' \ar[ull]^(.7){p_{12}'} \ar[dd]
\ar[rr]&&V_2' \ar[uu]_{p_{2}'} \ar[dd]^(.35){}\\
A\ar[dr]_{f}&&V_1 \ar[ll]_{p_{1}} %
\ar[rr]_(.35){}|!{[r];[r]}\hole\ar[dr]^{\mu_1}&&V \ar[dr]^{\mu}&\\
&A'&&V_1'\ar[ll]^{p_{1}'} \ar[rr]^{}&&V' .
} }
\end{equation}
 
If $f$ is flat, $g$ and $h$ are proper, $V_1= \mu\inverse(V_1')$, and the
diagram
\[ 
\small
\xymatrix{ B\ar[r]^{g} & B' \\
  V_2 \ar[u]^{p_{2}} \ar[r]^{\mu_2} &V_2' \ar[u]_{p_2'}}
\]
is cartesian, then $(\mu, f,g,h)$ is a \emph{proper} morphism. If $f$, $g$,
and $h$ are smooth, $V_1= \mu\inverse(V_1')$, $V_2= \mu\inverse(V_2')$, and
the diagram
\[
\small
\xymatrix@1{C \ar[r]^-{h} & C' \\
  V_1\cap V_2 \ar[r]^-{\mu_{12}} \ar[u]^-{p_{12}} &V_1'\cap V_2'
  \ar[u]^-{p_{12}'} }
\]
is cartesian, then $(\mu, f,g,h)$ is a \emph{smooth} morphism.

The next proposition is a restatement of \cite[\S6.5, 6.6, and
6.7]{lusztig:bases}.

\begin{proposition}\label{pro:convlin}
  Suppose $\CF'$ is in $K^H(A')$, $\CG'$ is in $K^H(B')$, $\CF$ is in
  $K^H(A)$, and $\CG$ is in $K^H(B)$.
  \begin{enumerate}
  \item If $(\mu, f,g,h)$ is proper, then $h_* \left( f^* \CF' \star \CG
    \right) =\CF'\star' g_*\,\CG $ in $K^H(C')$. \label{i:conv1}
  \item If $(\mu, f,g,h)$ is smooth, then $h^* \left( \CF' \star' \CG'
    \right) =f^*(\CF')\star g^*(\CG') $ in $K^H(C)$. \label{i:conv2}
  \item If $B=V_2=V$, $p_2=\id$, and $\BBC_V$ is the trivial line bundle on
    $V$, then $\CF \star [\BBC_V]= (p_{12})_* p_1^* (\CF) $ in
    $K^{H}(C)$. \label{i:conv3}
  \end{enumerate}
\end{proposition}

\subsection{Applications} 

In this subsection we apply the product constructions above in the
setting of Steinberg varieties.

As a first application we consider the $K^{\Gbar}(Z)$-module structure on
$K^{\Gbar}(\XEJ)$ used in the construction of $\psi^{IJ}$. If $Y$ is a
subvariety of $Z$, let
\[
\Ytilde=\{\, ((x, gB), (x, hB))\mid (x, gB, hB)\in Y\,\}
\]
denote the image of $Y$ in $\FNt\times \FNt$ under the natural embedding
$Z\to \FNt\times \FNt$.

As in~\cite[\S7.9]{lusztig:bases}, it is easy to see that taking $V=\FNt^3$,
$V_1= \Ztilde\times \FNt$, $V_2= \FNt \times \Ztilde$, and using the obvious
projections $V_1\to Z$ and $V_2\to Z$ in~\eqref{eq:tor}, we obtain a
``convolution product''
\[
\star\colon K^{\Gbar}(Z) \times K^{\Gbar}(Z) \to K^{\Gbar}(Z)
\]
that defines an $R(\Gbar)$-algebra structure on $K^{\Gbar}(Z)$.

Define $\FNt^J=\{\, (x, gP_J)\in \FN \times G/P_J \mid g\inverse x\in
\fp_J\,\}$ and
\[
\Xtilde^{\emptyset J} =\{\, ((x, gB), (x, hP_J))\in \FNt \times (\fg \times
G/P_J)\mid (x, gB, hB)\in \XEJ\,\}.
\]
Then taking $V= \FNt^2 \times (\fg\times G/P_J)$, $V_1= \Ztilde\times (\fg
\times G/P_J)$, $V_2= \FNt \times \Xtilde^{\emptyset J}$, $p_{1}$ to be the
composition $V_1\to \Ztilde \to Z$, $p_{2}$ to be the composition $V_2\to
\Xtilde^{\emptyset J} \to \XEJ$, and defining $p_{12}\colon V_1\cap V_2\to
\XEJ$ by $p_{12} ((x,gB), (x,hB), (x,kP_J))= (x, gB, kP_J)$
in~\eqref{eq:tor}, we obtain an $R(\Gbar)$-module linear map
\[
\star_J\colon K^{\Gbar}(Z) \times K^{\Gbar}(\XEJ) \to K^{\Gbar}(\XEJ)
\]
that defines a $K^{\Gbar}(Z)$-module structure on $K^{\Gbar}(\XEJ)$.

Define $\eta\colon \FNt^3\to \FNt^2\times (\fg\times G/P_J)$ by 
\[
\eta((x, gB), (y, hB), (z, kB)) = ((x, gB), (y, hB), (z, kP_J)).
\]
Then as in diagram~\eqref{eq:dmap}, the tuple $(\eta, \id, \etaEJ, \etaEJ)$
defines a proper morphism from the diagram defining the multiplication on
$K^{\Gbar}(Z)$ to the diagram defining the $K^{\Gbar}(Z)$-module structure
on $K^{\Gbar}(\XEJ)$. The next proposition follows
from~\autoref{pro:convlin}\,\autoref{i:conv1}.

\begin{proposition}\label{pro:xlin}
  Suppose $J$ is a subset of $S$. There is a left $K^{\Gbar}(Z)$-module
  structure on $K^{\Gbar}(\XEJ)$ such that $\etaEJ_* \colon K^{\Gbar}(Z) \to
  K^{\Gbar}(\XEJ)$ is $K^{\Gbar}(Z)$-linear.
\end{proposition}

As a second application we consider $K^{\Gbar}(Z_1)$-module structures
on $K^{\Gbar}(Z_{\ssleq w})$ and $K^{\Gbar}(Z_{w})$ for $w$ in
$W$. Notice that $Z_1=\{\, (x, gB, gB)\in Z \mid g\inverse x\in \fu\}$
is a closed subvariety of $Z$ isomorphic to $\FNt$. Thus, $Z_1$ is
smooth and so $K^{\Gbar}(Z_1)$ has a natural ring structure given by
tensor product.

Suppose $w$ is in $W$. Taking $V= \FNt^3$, $V_1= \Ztilde_1\times \FNt$,
$V_2= \FNt \times \Ztilde_{\ssleq w}$, and using the obvious projections
$V_1\to Z_1$ and $V_2\to Z_{\ssleq w}$ in~\eqref{eq:tor}, we obtain a
$K^{\Gbar}(Z_1)$-module structure on $K^{\Gbar}(Z_{\ssleq w})$ that is
denoted by $\star_w$,
\[
\star_w\colon K^{\Gbar}(Z_1) \times K^{\Gbar}(Z_{\ssleq w}) \to
K^{\Gbar}(Z_{\ssleq w}).
\]
Notice that when $w=w_0$ is the longest element in $W$ we get a
$K^{\Gbar}(Z_1)$-module structure on $K^{\Gbar}(Z)$. To simplify the
notation slightly, set $\star_1=\star_{w_0}$.  It follows
from~\autoref{pro:convlin}\,\autoref{i:conv1} that this module structure is
compatible, via the inclusion $j_1\colon Z_1\to Z$, with the multiplication
$\star$ in $K^{\Gbar}(Z)$, in the sense that
\begin{equation}
  \label{eq:star1}
  \CF\star_1 \CG= (j_1)_*\CF \star \CG
\end{equation}
for $\CF$ in $K^{\Gbar}(Z_1)$ and $\CG$ in $K^{\Gbar}(Z)$.

Similarly, taking $V=\FNt\times \Ztilde_w$, $V_1= (\Ztilde_1\times \FNt)
\cap (\FNt \times \Ztilde_{ w})$, $V_2= \FNt \times \Ztilde_{ w}$, and using
the obvious projections $V_1\to Z_1$ and $V_2\to Z_w$, we obtain a
$K^{\Gbar}(Z_1)$-module structure on $K^{\Gbar}(Z_w)$ that is
denoted by $\star_w'$,
\[
\star_w'\colon K^{\Gbar}(Z_1) \times K^{\Gbar}(Z_w) \to K^{\Gbar}(Z_w).
\]

Now suppose $y$ and $w$ are in $W$ with $y\leq w$. Let $j_y^w\colon
Z_{\ssleq y}\to Z_{\ssleq w}$ be the inclusion. In the special case when
$w=w_0$ set $j_y=j_y^{w_0}$. Then $j_y$ is the inclusion of $Z_{\ssleq y}$
in $Z$. Recall that $r_w\colon Z_w\to Z_{\ssleq w}$ is the inclusion. Notice
that $j_y^w$ is a closed embedding and that $r_w$ is an open embedding. The
proof of the next lemma is a straightforward application
of~\autoref{pro:convlin}.

\begin{lemma}\label{lem:starlin}
   Suppose $y$ and $w$ are in $W$ with $y\leq w$. Then the maps
   \[
   (j_y^w)_*\colon K^{\Gbar}(Z_{\ssleq y}) \to K^{\Gbar}( Z_{\ssleq w})
   \quad \text{and}\quad r_w^* \colon K^{\Gbar}(Z_{\ssleq w}) \to K^{\Gbar}
   (Z_w)
   \]
   are $K^{\Gbar}(Z_1)$-module homomorphisms.
\end{lemma}

\section{The \texorpdfstring{$K^{\Gbar}(Z)$}{}-module homomorphism
  \texorpdfstring{$\eta^{\emptyset J}_* \colon K^{\Gbar}(Z)\to
    K^{\Gbar}(\XEJ)$}{}} \label{sec:as}

In this section we prove~\autoref{thm:as}, the assertion that
$\etaEJ_*(\mathbf a_s)= 0$ for $s$ in $J$, and thus complete the argument
that there is an $R(\Gbar)$-module homomorphism $\psi^{IJ}\colon \CHIJ\to
K^{\Gbar}(\XIJ)$ such that $\psi^{IJ} \chi^{IJ}= \etaIJ \varphi$ as
in~\autoref{thm:main}.

For a subset $I$ of $S$, define $\CO_I$ to be the orbit $G\cdot (B, w_IB)$
in $G/B \times G/B$ and define $Z_I=Z_{w_I}$ in $Z$.
 
\begin{proposition}
  Suppose $I$ is a subset of $S$.
  \begin{enumerate}
  \item The closure of $\CO_I$ in $G/B\times G/B$ is
    \[
    \OIbar= \coprod_{w\in W_I} \CO_w = \{\, (gB, hB) \mid g\inverse h\in
    P_I\,\}.
    \] \label{i:1}
  \item The closure of $Z_I$ in $Z$ is  
    \begin{align*}
      \ZIbar&=\{(x, gB, hB)\in Z\mid (gB,hB)\in \OIbar,\, g\inverse x\in
      \fu_I,\, h\inverse x\in \fu_I\,\}\\
      &=\{(x, gB, hB)\in Z\mid g\inverse h\in P_I,\, g\inverse x\in
      \fu_I\,\}.
    \end{align*} \label{i:2}
  \item There are isomorphisms $\OIbar \cong G\times^{B} P_I/B$ and $
    \ZIbar \cong G\times^{B} (\fu_I\times P_I/B)$ such that the
    diagram
    \begin{equation}
      \label{eq:new}
      \vcenter{\vbox{
          \xymatrix{ G\times^{B} (\fu_I\times P_I/B)\ar[r]_-{\cong} \ar[d] &
            \ZIbar \ar[d]^{q_I} \\
            G\times^{B} P_I/B \ar[r]_-{\cong} & \OIbar }    
        }}
    \end{equation}
    commutes. Here, the left vertical map is the obvious one and $q_I$ is
    the projection on the second and third coordinates. In particular, the
    varieties $\OIbar$ and $\ZIbar$ are smooth and the projection $q_I
    \colon \ZIbar\to \OIbar$ is a vector bundle with fibre
    $\fu_I$. \label{i:3}
  \end{enumerate}
\end{proposition}

\begin{proof}
  The first equality in~\autoref{i:1} is well known. For the second equality,
  set
  \[
  Y_1= \{\, (gB, hB) \mid g\inverse h\in P_I\,\}.
  \] 
  Given $g$ in $G$, $(gB, gw_IB)$ is obviously in $Y_1$ so $\CO_I\subseteq
  Y_1$. Thus $\OIbar\subseteq Y_1$. Conversely, if $(gB, hB)$ is in $Y_1$,
  then there is a $w$ in $W_I$ so that $g\inverse h$ is in $BwB$. Say
  $g\inverse h=b_1wb_2$. Then $(gB, hB)= (gB, gb_1 w B)$ is in $\CO_w$. This
  proves the first statement.

  The second equality in~\autoref{i:2} follows from~\autoref{i:1}. Set
  \[
  Y_2= \{(x, gB, hB)\in Z\mid g\inverse h\in P_I,\, g\inverse x\in
  \fu_I\,\}.
  \] 
  Then $Y_2$ is closed in $Z$ and the projection $Y_2\to \OIbar$ is a vector
  bundle with fibre $\fu_I$. It follows that $Y_2$ is irreducible. Since
  \[
  Z_I=\{\, (x, gB, gw_IB) \mid g\inverse x\in \fu\cap w_I\fu\,\}
  \]
  and $\fu_I=\fu\cap w_I\fu$, we see that $Z_I$, and hence $\ZIbar$, is
  contained in $Y_2$. Since $\ZIbar$ is an irreducible component of $Z$ and
  $Y_2$ is closed and irreducible we have $\ZIbar= Y_2$. This proves the
  second statement.

  For each of the varieties $\OIbar$ and $\ZIbar$, projection on the
  first factor of $G/B$ is a $G$-equivariant surjection and induces
  $G$-equivariant isomorphisms
  \[
  \OIbar \cong G\times^{B} P_I/B\qquad \text{and} \qquad \ZIbar \cong
  G\times^{B} (\fu_I\times P_I/B).
  \]
  It is straightforward to check that diagram~\eqref{eq:new} commutes.  This
  proves the last statement. \qed
\end{proof}

Suppose $s$ is in $S$ and consider the vector bundle $q_I\colon \ZIbar \to
\OIbar$ in the case when $I=\{s\}$.  Let $i_s\colon \Osbar\to G/B\times G/B$
and $k_s\colon \Zsbar\to Z$ be the inclusions and set $q_s=q_{\{s\}}$. Then
the diagram of $\Gbar$-equivariant morphisms
\begin{equation*}
  \vcenter{\vbox{
      \xymatrix{ \Zsbar \ar[d]^{q_s} \ar[r]^-{k_s} & Z \ar[r]^-{\ktilde} &
        \FNt \times G/B \ar[d]^{q\times \id} \\  
        \Osbar \ar[rr]^-{i_s} && G/B \times G/B,}  
    }}
\end{equation*}
where $\ktilde$ is the inclusion, commutes. Given $\lambda', \lambda''$ in
$X(T)$, set
\[
\CL_{\lambda', \lambda''}= (k_s)_* q_s^* i_s^* \left( [\CL_{\lambda'}
  \boxtimes \CL_{\lambda''}] \right) = (k_s)_* k_s^*  \ktilde^* (q\times \id)^*
\left( [\CL_{\lambda'} \boxtimes \CL_{\lambda''}] \right)
\]
in $K^{\Gbar}(Z)$.

\begin{proposition}\label{pro:special}
  Suppose $J\subseteq S$, $\alpha\in \Pi_J$, $s=s_\alpha$, $\lambda''\in
  X(T)$ with $\langle \lambda'', \check \alpha \rangle =-1$, and
  $\lambda'\in X(T)$. Then $\etaEJ_* \left( [\CL_{\lambda',\lambda''}]
  \right) =0$ in $K^{\Gbar}(\XEJ)$.
\end{proposition}

\begin{proof}
  Set $P_s=P_{\{s\}}$, $\fu_s=\fu_{\{s\}}$, and define
  \[
  Y_s= \eta^{\emptyset \{s\}}(\Zsbar)= \{\, (x, gB, gP_s)\in \FN
  \times G/B \times G/P_s \mid g\inverse x\in \fu_{s} \,\}.
  \]
  We have a commutative diagram of proper morphisms,
  \[
  \xymatrix{ G\times^{B} (\fu_{s} \times P_{s}/B) \ar[r]_-{\cong}
    \ar[d]&  \Zsbar \ar[d]^{\eta'} \ar[r]^-{k_s} & Z \ar[d]^{\etaEJ} \\
    G\times^{B} \fu_{s}\ar[r]_-{\cong} & Y_s \ar[r]^{\eta^J} & \XEJ,}
  \]
  where the isomorphism $G\times^{B} \fu_s \xrightarrow{\ \cong\ }
  Y_s$ is induced by projection on the second factor, $\eta'(x, gB,
  hB)= (x, gB, hP_s)$, and $\eta^J(x, gB, hP_s)= (x, gB,
  hP_J)$. Notice in particular that $\eta'$ is a Zariski locally
  trivial $\BBP^1$-bundle.

  Set $\CL^s_{\lambda', \lambda''}= q_s^* i_s^* \left( \CL_{\lambda'}
    \boxtimes \CL_{\lambda''} \right)$. To prove the proposition it is
  enough to show that $\eta'_*([\CL^s_{\lambda', \lambda''}])= 0$ in
  $K^{\Gbar}(Y_s)$. We show that $R^n \eta'_* \CL^s_{\lambda',
    \lambda''}|_y=0$ for all $y$ in $Y_s$ and for all $n\geq 0$. It
  then follows that $R^n \eta'_* \CL^s_{\lambda', \lambda''}=0$ for
  all $n\geq 0$ and so $\eta'_*([\CL^s_{\lambda', \lambda''}])= 0$ by
  the definition of $\eta'_*$.

  Suppose $y$ is in $Y_s$. Because $\eta'$ is $G$-equivariant we may
  assume $y= (x, B, P_s)$ for some $x$ in $\fu_s$. Define
  $E=(\eta')\inverse(y)= \{\, (x, B, pB)\mid p\in P_s\,\}$ and let
  $f\colon E\to \Zsbar$ be the inclusion mapping. Projection on the
  third factor defines an isomorphism $E \cong P_s/B$ and so $E\cong
  \BBP^1$.

  The restriction of $\CL^s_{\lambda',\lambda''}$ to $E$ is given by
  \[
  \CL^s_{\lambda',\lambda''} |_{E} = (i_sq_sf)^* (\CL_{\lambda'} \boxtimes
  \CL_{\lambda''})= (p_1i_sq_sf)^* \CL_{\lambda'} \otimes (p_2i_sq_sf)^*
  \CL_{\lambda''}
  \]
  where $p_1$ and $p_2$ are the two projections of $G/B \times G/B$
  onto $G/B$. The composition $p_1 i_s q_s f$ is a constant map, so
  $\CL_{\lambda'}$ pulls back to the trivial line bundle
  $\CO_{\BBP^1}(0)$ on $E$. The map $p_2 i_s q_s f$ restricts to an
  isomorphism $E \cong P_s/B$, and because $\langle \lambda'', \check
  \alpha \rangle =-1$, $\CL_{\lambda''}$ pulls back to the
  tautological bundle on $P_s/B$.  Therefore,
  $\CL^s_{\lambda',\lambda''} |_{E}$ may be identified with the line
  bundle $\CO_{\BBP^1}(-1)$ on $\BBP^1$.

  Finally, for $n\geq 0$, $H^n( E, \CL^s_{\lambda', \lambda''} |_{E}) \cong
  H^n( \BBP^1, \CO_{\BBP^1}(-1)) =0$, and hence by Grauert's Theorem, $R^n
  \eta'_* \CL^s_{\lambda', \lambda''}|_y\cong H^n( E, \CL^s_{\lambda',
    \lambda''} |_{E}) =0$ as claimed. \qed
\end{proof}

We can now complete the proof of~\autoref{thm:as}. With $J$, $\alpha$, and
$s$ as above, suppose that $\lambda'$ and $\lambda''$ are in $X(T)$,
$\langle\lambda ', \check \alpha \rangle= \langle\lambda '', \check \alpha
\rangle= -1$, and $-\alpha=\lambda'+\lambda''$. Define $\mathbf a_s$ in
$K^{\Gbar}(Z)$ by
\[
\mathbf a_s= v\inverse \, \CL_{\lambda', \lambda''} = v\inverse (k_s)_*
q_s^* i_s^* \left( [\CL_{\lambda'} \boxtimes \CL_{\lambda''}] \right).
\]
Lusztig \cite[\S7.19, 7.25]{lusztig:bases} has shown that $\mathbf a_s$ does
not depend on the choice of $\lambda'$ and $\lambda''$. It follows
from~\autoref{pro:special} that $\etaEJ_*(\mathbf a_s)= 0$, as asserted in
the statement of the theorem.

\section{An \texorpdfstring{$A$}{}-basis of
  \texorpdfstring{$\CHIJ$}{}} \label{sec:basis}

By~\autoref{lem:double}, $\{\, t_\lambda z\mid z\in \WIJ,\, \lambda \in
X(T)_{z} ^+ \,\}$ is a set of $(W_I,W_J)$-double coset representatives in
$\Wex$. In this section we prove~\autoref{thm:hijbasis} and thus show that
$\CHIJ$ has an $A$-basis indexed by $W_I\backslash \Wex /W_J$.

We need to show that the set $\CE=\{\, C_{w_I} \theta_\lambda T_z C_{w_J}\mid
z\in \WIJ,\, \lambda \in X(T)_{z} ^+ \,\}$ is an $A$-basis of $\CHIJ$.  To
see that $\CE$ is linearly independent, notice that the specialization
$v\mapsto 1$ defines an $A$-linear homomorphism
\[
\CHIJ \to \epsilon_I \BBZ[\Wex] \epsilon_J,
\]
where as above, $\epsilon_I=\sum_{w\in
  W_I}\epsilon_ww$. By~\autoref{lem:double}, $\{\, \epsilon_I t_\lambda z
\epsilon_J \mid z\in \WIJ,\, \lambda \in X(T)_{z} ^+ \,\}$ is linearly
independent in $\epsilon_I \BBZ[\Wex] \epsilon_J$ and so $\CE$ is linearly
independent in $\CHIJ$.

To prove that $\CE$ spans $\CHIJ$ we adapt an argument of Nelsen and Ram
\cite{nelsenram:kostka}.  It follows from the relation
\begin{equation}
  \label{eq:8}
  \theta_\lambda T_s -T_s \theta_{s(\lambda)}= (v^2-1) \frac
  {\theta_\lambda-\theta_{s(\lambda)}} {1-\theta_{-\alpha}},    
\end{equation}
for $\lambda$ in $X(T)$ and $s=s_\alpha$ with $\alpha$ in $\Pi$ that
if $w$ is in $W$, then $\theta_\lambda T_w$ is in the $A$-span of
$\{\, T_y \theta_\mu\mid y\leq w,\ \mu\in X(T)\,\}$. Therefore, if $w$
is in $W$ and $w=w_1zw_2$ where $z$ is in $\WIJ$, $w_1$ in $W_I$,
$w_2$ in $W_J$, and $\ell(w)= \ell(w_1)+\ell(z)+\ell(w_2)$, then
$\theta_\lambda T_{w_1zw_2}=\theta_\lambda T_{w_1} T_z T_{w_2}$ is an
$A$-linear combination of elements of the form $T_y \theta_\mu T_z
T_{w_2}$ where $y\leq w_1$ is in $W_I$. Therefore, the elements
$C_{w_I} \theta_\lambda T_z C_{w_J}$ for $z$ in $\WIJ$ and $\lambda$
in $X(T)$ span $\CHIJ$.

Fix $z$ in $\WIJ$ and suppose that $s=s_\alpha$ with $\alpha$ in $\Pi_I \cap
z\Pi_J$. Then $T_sT_z=T_{sz}=T_{z z\inverse sz}= T_z T_{s'}$, where
$s'=z\inverse sz$ is in $J$. Thus, multiplying~\eqref{eq:8} on the right
by $T_z$ and then taking the image in $\CHIJ$, we have
\begin{equation}
  \label{eq:9}
  C_{w_I} \left( -\theta_\lambda T_z +\theta_{s(\lambda)}T_z \right)
  C_{w_J} = C_{w_I} \left( (v^2-1) \frac
    {\theta_\lambda-\theta_{s(\lambda)}} {1-\theta_{-\alpha}} T_z \right)
  C_{w_J}  
\end{equation}
for any $\lambda$ in $X(T)$. Replacing $\lambda$ by $\lambda-\alpha$,
subtracting the result from~\eqref{eq:9}, and simplifying, we get
\begin{equation}
  \label{eq:x4}
  C_{w_I} \theta_{s(\lambda)} T_z C_{w_J}  = 
  C_{w_I} \left( v^2 \theta_{\lambda} - \theta_{\lambda-\alpha} + 
    v^2\theta_{s(\lambda) +\alpha} \right) T_z C_{w_J}. 
\end{equation}
The derivation of~\eqref{eq:x4} from~\eqref{eq:8} is the same as in the
proof of \cite[Proposition 2.1]{nelsenram:kostka}. After relabeling the
elements in $\CH$ and replacing $t$ by $v^2$,~\eqref{eq:x4} is the same as
the recursion formula at the end of the proof of \cite[Proposition
2.1]{nelsenram:kostka}. It follows that if $\lambda$ is in $X(T)$ with
$d=\langle \lambda, \check \alpha \rangle \geq 0$, then there are
polynomials $p_0$, $p_1$, \dots, $p_{\lfloor d/2 \rfloor}$ in $\BBZ[v^2]$
such that
\begin{equation}
  \label{eq:x5}
  C_{w_I} \theta_{s(\lambda)} T_zC_{w_J}  = C_{w_I}
  \bigg( \sum_{j=0}^{\lfloor d/2\rfloor} p_j \theta_{\lambda -j\alpha} 
  \bigg) T_z C_{w_J}. 
\end{equation}
The polynomials $p_j$ do not depend on $\lambda$ and are given by the
formula in \cite[Proposition 2.1]{nelsenram:kostka} with $t$ replaced by
$v^2$.

Now if $\mu$ is in $X(T)$ and $\langle \mu ,\check \alpha \rangle <0$, then
using~\eqref{eq:x5} with $s(\lambda)=\mu$ we see that $C_{w_I} \theta_\mu
T_z C_{w_J}$ may be rewritten as an $A$-linear combination of terms $C_{w_I}
\theta_{\mu'} T_z C_{w_J}$ where $\langle \mu', \check \alpha \rangle \geq
0$ and $\langle \mu' ,\check \beta \rangle \geq \langle \mu ,\check \beta
\rangle$ for all $s_\beta$ with $\beta$ in $\Pi$. It follows that $\CE$
spans $\CHIJ$, as claimed.

\section{Equivariant \texorpdfstring{$K$}{}-theory of
  \texorpdfstring{$\XIJ$}{}} \label{sec:wideiso}

In this section we compute $K^{\Gbar}(\XIJ)$ in terms of $\CHIJ$, using
$\CH$ and $K^{\Gbar}(Z)$, and thus complete the proof
of~\autoref{thm:main}. We begin with $K^{\Gbar}(\XIJ_z)$.

\subsection{\texorpdfstring{$K^{\Gbar}(\XIJ_z)$}{} and the Springer
  resolution of \texorpdfstring{$\SN_z$}{}} \label{ssec:xijz}

In this subsection we prove~\autoref{thm:ost}.

Suppose $z$ is in $\WIJ$. We need to show that there is a commutative
diagram of $R(\Gbar)$-modules
\begin{equation*}
  \xymatrix{ K^{\Gbar}(Z_z) \ar[r]_{\cong} \ar[d]^{\eta^z_*} &
    K^{\Lbar_z}(\SNt_z) \ar[d]^{(p_z)_*} \\ 
    K^{\Gbar}(\XIJ_z) \ar[r]_{\cong} & K^{\Lbar_z}(\SN_z), }  
\end{equation*}
where the horizontal maps are isomorphisms, $p_z$ is the Springer resolution
of the nilpotent cone $\SN_z$ of $\fl_z$, and the vertical maps are
surjections.

Set
\begin{equation*}
  P_z=P_I\cap {}^zP_J  \quad \text{and} \quad V_z=(L_I\cap {}^zU_J)
  (U_I\cap {}^zL_J) (U_I\cap {}^zU_J).  
\end{equation*}
Then 
\begin{equation}
  \label{eq:pz}
  P_z= L_z V_z=L_z (L_I\cap {}^zU_J) (U_I\cap {}^zL_J) (U_I\cap {}^zU_J),  
\end{equation}
$L_z$ is a Levi subgroup of $P_z$, and $V_z$ is the unipotent radical of
$P_z$ (see~\cite[\S69B] {curtisreiner:methodsII}). The Lie algebra analog of
the factorization~\eqref{eq:pz} is
\begin{equation}
  \label{eq:qf}
  \fp_z= \fl_z +\fv_z=\fl_z+ (\fl_I\cap z\fu_J) +(\fu_I\cap z\fl_J)+
  (\fu_I\cap z\fu_J).  
\end{equation}
Set $\fu_z=\fu\cap \fl_z$. It follows from~\eqref{eq:qf} that $\fu\cap
z\fu =\fu_z+\fv_z$. 

The rule $(x, gB, gzB)\mapsto (x, gP_I, gzP_J)$ defines a surjective,
$G$-equivariant morphism from $Z_z$ onto the $G$-orbit of $(P_I, zP_J)$ in
$G/P_I \times G/P_J$. The fibre over $(P_I, zP_J)$ is
\begin{align*}
  F&=\{\, (x, pB, pzB)\in Z_z \mid p\in P_z,\ p\inverse
  x \in \fu \cap z\fu\,\} \\
  &= \{\, (x, hB, hzB)\in Z_z \mid h\in L_z,\ h\inverse
  x \in \fu_z+\fv_z\,\}.
\end{align*}
Thus, $Z_z\cong G\times^{P_z}F$. Projection onto the second and third
factors is a surjective, $P_z$-equivariant morphism from $F$ onto the
$P_z$-orbit of $(B, zB)$ in $G/B \times G/B$ with fibre over $(B, zB)$
isomorphic to $\fu_z+\fv_z$. Thus, $F\cong P_z\times^{B\cap {}^zB}
(\fu_z+\fv_z)$.  Define
\[
\ptilde\colon F\to \SNt_z\quad \text{by}\quad \ptilde(h(x+v),hB, hzB)=(hx,
hB_z),
\]
for $x$ in $\fu_z$, $v$ in $\fv_z$, and $h$ in $L_z$, and define
\[
i_z\colon \SNt_z\to F\quad\text{by}\quad i_z(x, hB_z)= (x, hB, hzB).
\]
Then $\ptilde$ and $i_z$ are well-defined morphisms that may be identified
with the morphisms $P_z\times^{B\cap {}^zB} (\fu_z+\fv_z)\to L_z\times^{B_z}
\fu_z$ and $L_z\times^{B_z} \fu_z \to P_z\times^{B\cap {}^zB} (\fu_z+\fv_z)$
given by $hu *(x+v)\mapsto h*x$ and $h*x\mapsto h*x$ for $h$ in $L_z$, $u$
in $V_z$, $x$ in $\fu_z$, and $v$ in $\fv_z$. Thus, $\ptilde \colon F\to
\SNt_z$ is an $\Lbar_z$-equivariant vector bundle with fibre $\fv_z$ and
zero section $i_z$.

Next, the projection of $\XIJ_z$ onto the second and third factors is a
surjective, $G$-equivariant morphism onto the $G$-orbit of $(P_I, zP_J)$ in
$G/P_I \times G/P_J$. The fibre over $(P_I, zP_J)$ is
\begin{align*}
  F^{IJ}&=\{\, (x, P_I, zP_J)\in \XIJ_z \mid x \in \SN_z+\fv_z \,\}.
\end{align*}
Thus, $\XIJ_z \cong G\times^{P_z} (\SN_z+\fv_z)$ and $F^{IJ}\cong
\SN_z+\fv_z$. Define
\[
\ptilde^{IJ}\colon F^{IJ}\to \SN_z\quad \text{by}\quad \ptilde^{IJ}(x+v,P_I,
zP_J)=x,
\]
for $x$ in $\SN_z$ and $v$ in $\fv_z$, and define
\[
i_z^{IJ}\colon \SN_z\to F^{IJ} \quad \text{by}\quad i_z^{IJ}(x)= (x, P_I,
zP_J).
\]
Obviously the morphism $\ptilde^{IJ}$ is an $\Lbar_z$-equivariant vector
bundle with fibre $\fv_z$ and zero section $i_z^{IJ}$.

Finally, let $\tilde \eta^z\colon F\to F^{IJ}$ be the restriction of
$\eta^z$. Then the diagram
\begin{equation}
  \label{eq:co1}
  \vcenter{\vbox{
      \xymatrix{ \SNt_z\ar[r]^-{i_z} \ar[d]^{p_z} & F\ar@{^{(}->}[r]
        \ar[d]^{\tilde \eta^z} & Z_z\ar[d]^{\eta^z}\\ 
        \SN_z\ar[r]^-{i_z^{IJ}} & F^{IJ}\ar@{^{(}->}[r] &\XIJ_z}  
    }}
\end{equation}
commutes. Moreover, one checks that the left-hand square is cartesian.

Now applying $K^{\Gbar}$ and $K^{\Lbar_z}$ to~\eqref{eq:co1} we get
\[
\xymatrix{K^{\Gbar}(Z_z) \ar[r]^-{\res_z}\ar[d]^{\eta^z_*} &
  K^{\Lbar_z}(F) \ar[r]^-{i_z^*} \ar[d]^{\tilde \eta^z_*} &
  K^{\Lbar_z}(\SNt_z) \ar[d]^{(p_z)_*} \\
  K^{\Gbar}(\XIJ_z) \ar[r]^-{\res_z^{IJ}} & K^{\Lbar_z}(F^{IJ})
  \ar[r]^-{(i_z^{IJ})^*} & K^{\Lbar_z}(\SN_z), }
\]
where $\res_z=\res_{F}$ is defined using the isomorphism $Z_z\cong
G\times^{P_z}F$ and $\res_z^{IJ}=\res_{F^{IJ}}$ is defined using the
isomorphism $\XIJ_z\cong G\times^{P_z}F^{IJ}$. The left-hand square commutes
by the naturality of $\res$. For the right-hand square, the diagram
\[
\xymatrix{F \ar[r]^-{\tilde p} \ar[d]^{\tilde \eta^z} & \SNt_z
  \ar[d]^{p_z} \\
  F^{IJ} \ar[r]^-{\tilde p^{IJ}}& \SN_z}
\]
is cartesian and $\tilde p$ and $\tilde p^{IJ}$ are vector bundles, so
$\tilde \eta^z_* \tilde p^* = (\tilde p^{IJ})^* (p_z)_*$. By the Thom
isomorphism in equivariant $K$-theory, $i_z^*= (\tilde p^*)\inverse$ and
$(i_z^{IJ})^* = ((\tilde p^{IJ})^*)\inverse$, so $(i_z^{IJ})^* \tilde
\eta^z_*= (p_z)_* i_z^*$. Finally, by~\autoref{lem:ostlem} $(p_z)_*$ is a
surjection. Therefore, $\tilde \eta^z_*$ and $\eta^z_*$ are surjections as
well.

\subsection{The sequence (\ref{eq:exz}) is exact} \label{ssec:ex}

\begin{proposition}\label{pro:k1}
  Suppose $H$ is a linear algebraic group and that $Y$ is an $H$-variety
  such that $H$ acts on $Y$ with finitely many orbits. Then $K^H_1(Y)=0$.
\end{proposition}

\begin{proof}
  If $H$ acts transitively with point stabilizer $H_0$, then $Y\cong
  H/H_0$ and the result is known (see~\cite[\S1.3
  (p)]{kazhdanlusztig:langlands}). In the general case, choose an open
  orbit $\CO$ in $Y$. Then there is an exact sequence
  \[
  \dotsm \to K^H_1(Y\setminus \CO) \to K^H_1(Y) \to K^H_1(\CO) \to \dotsm .
  \]
  By induction on the number of orbits, $K^H_1(Y\setminus \CO) =0$. We have
  already observed that $K^H_1(\CO)=0$. Thus, $K^H_1(Y)=0$. \qed
\end{proof}

\begin{lemma}\label{lem:k1}
  Suppose $I, J\subseteq S$, and $z\in \WIJ$. Then $K^{\Gbar}_1(\XIJ_z)=0$.
\end{lemma}

\begin{proof}
  The constructions used in the proof of~\autoref{thm:ost} apply to the
  functors $K_i^{\Gbar}$ for $i\geq 0$ (see \cite[\S5.2,
  5.4]{chrissginzburg:representation}) and give isomorphisms
  \[
  \xymatrix{ K_1^{\Gbar}(\XIJ_z) \ar[r]^-{\res_z^{IJ}}_-{\cong} &
    K_1^{\Lbar_z}(F^{IJ}) \ar[r]^-{(i_z^{IJ})^*}_-{\cong} &
    K_1^{\Lbar_z}(\SN_z) .}
  \]
  Because $\Lbar_z$ acts on $\SN_z$ with finitely many orbits, it follows
  from~\autoref{pro:k1} that $K^{\overline{L_{z}}}_1 (\SN_z)=0$. \qed
\end{proof}

The fact that sequence~\eqref{eq:exz} is exact follows immediately
from~\autoref{lem:k1} and the long exact sequence in equivariant $K$-theory.

\subsection{The isomorphism \texorpdfstring{$\CHIJ_{\ssleq z}\cong
    K^{\Gbar}(\XIJ_{\ssleq z})$}{}} \label{ssec:isoz1}

The rest of this section is devoted to the proof
of~\autoref{thm:wideiso2}. We first define the maps in~\eqref{eq:wideiso},
\[
\xymatrix{%
  \CH_{\ssleq z} \ar[r]^-{} \ar[d]^{\chi^{\ssleq z}}& K^{\Gbar}(Z_{\ssleq
    z}) \ar[r]^-{r_z^*} \ar[d]^{\eta^{\ssleq z}_*} & K^{\Gbar}(Z_z) \ar[r]
  \ar[d]^{\eta^z_*} &  K^{\Lbar_z}(\SNt_z) \ar[d]^{(p_z)_*} \\
  \CHIJ_{\ssleq z} \ar[r] & K^{\Gbar}(\XIJ_{\ssleq z})
  \ar[r]^-{(r_z^{IJ})^*} & K^{\Gbar}(\XIJ_z) \ar[r] &
  K^{\Lbar_z}(\SN_z), }
\]
and show that the diagram commutes.

The middle square is induced by the cartesian diagram~\eqref{eq:cart} and
the right-hand square is as in~\autoref{thm:ost}. Both of these squares
commute.

\subsection{The left-hand square in diagram
  (\ref{eq:wideiso})} 

We observed after~\eqref{eq:8} that if $\lambda$ is in $X(T)$, $y$ is in
$\WIJ$, and $w$ is in the double coset $W_IyW_J$, then $C_{w_I}
\theta_\lambda T_{w}C_{w_J}$ is in the span of $\{\, C_{w_I} \theta_\mu T_y
C_{w_J} \mid \mu \in X(T)\,\}$. Therefore,
$\chi^{IJ}(\CH_{w})\subseteq\CHIJ_y$. It follows that $\chi^{\ssleq
  z}(\CH_{\ssleq z})\subseteq \CHIJ_{\ssleq z}$. In particular,
$\chi^{\ssleq z}\colon \CH_{\ssleq z}\to \CHIJ_ {\ssleq z}$ is defined.

Consider the commutative diagram
\begin{equation*}
  \xymatrix{%
    \CH_z\ar@{^{(}->}[r] \ar[d]^{\chi^z} &
    \CH_{\ssleq z} \ar[d]^{\chi^{\ssleq z}}\\   
    \CHIJ_z\ar@{^{(}->}[r]  & \CHIJ_{\ssleq z}, }  
\end{equation*}
where the horizontal maps are the inclusions and $\chi^z$ is the restriction
of $\chi^{IJ}$ to $\CH_z$. It is clear that $\chi^y(\CH_y)=\CHIJ_y$ for $y$
in $\WIJ$ and so $\chi^{\ssleq z}$ is surjective.

The left-hand square in diagram~\eqref{eq:wideiso} is
\begin{equation*}
  \xymatrix{%
    \CH_{\ssleq z} \ar[r]^-{\varphi_z} \ar[d]^{\chi^{\ssleq z}} &
    K^{\Gbar}(Z_{\ssleq z}) \ar[d]^{\eta^{\ssleq  z}_*} \\  
    \CHIJ_{\ssleq z} \ar[r]^-{\psi_z} & K^{\Gbar}(\XIJ_{\ssleq z}), }   
\end{equation*}
where $\varphi_z$ and $\psi_z$ are defined below. 

Recall that for $w$ in $W$ $j_w\colon Z_{\ssleq w}\to Z$ is the
inclusion. Similarly, let $j^{IJ}_z$ denote the inclusion $\XIJ_{\ssleq
  z}\to \XIJ$ for $z$ in $\WIJ$. The maps $\varphi_z$ and $\psi_z$ are the
restrictions of $\varphi$ and $\psi^{IJ}$, respectively, in the sense that
$(j_z)_* \circ \varphi_z$ is the restriction of $\varphi$ to $\CH_{\ssleq
  z}$ and $(j^{IJ}_z)_* \circ \psi_z$ is the restriction of $\psi^{IJ}$ to
$\CHIJ_{\ssleq z}$. In order to prove~\autoref{thm:wideiso2} we need a
formula for $\varphi_z$, and so we define $\varphi_z$ explicitly, show that
$(j_z)_* \circ \varphi_z$ is the restriction of $\varphi$ to $\CH_{\ssleq
  z}$, and then define $\psi_z$.

For $w$ in $W$, let $q_{w,1} \colon Z_w\to G/B$ by $q_{w,1}(x,gB, gwB)= gB$.
Then $q_{w,1}$ is a $\Gbar$-equivariant affine space bundle over $G/B$ and
so $q_{w,1}^*\colon K^{\Gbar}(G/B) \to K^{\Gbar}( Z_w)$ is an
$R(\Gbar)$-module isomorphism.

\begin{theorem}\label{thm:basis}
  Suppose $w$ is in $W$. There is an $A$-module isomorphism
  \[
  \varphi_w\colon \CH_{\ssleq w} \to K^{\Gbar}(Z_{\ssleq w})
  \]
  such that
  \begin{enumerate}
  \item $(j_w)_* \varphi_w\colon \CH_{\ssleq w}\to K^{\Gbar}(Z)$ is the
    restriction of $\varphi$ to $\CH_{\ssleq w}$, and \label{it:phi1}
  \item for $\lambda$ in $X(T)$, $r_w^* \varphi_w(\theta_\lambda T_w)=
    \epsilon_w q_{w,1}^* [\CL_\lambda]$. \label{it:phi2}
  \end{enumerate}
  In particular, $\{\, r_w^* \varphi_w(\theta_\lambda T_w)\mid \lambda\in
  X(T) \,\}$ is an $A$-basis of $K^{\Gbar}(Z_w)$.
\end{theorem}

\begin{proof}
  Lusztig has shown (see \cite[Lemma 8.9]{lusztig:bases}) that there is a
  unique element $\xi_y$ in $K^{\Gbar}(Z_{\ssleq y})$ such that
  $(j_y)_*(\xi_y)=\varphi(T_y)$.  Let $d_1\colon \FNt\to Z_1$ be the
  ``diagonal'' isomorphism given by $d_1(x, gB)= (x, gB, gB)$. With this
  notation, using the $K^{\Gbar}(Z_1)$-module structure $\star_w$ on
  $K^{\Gbar}(Z_{\ssleq w})$, define
  \[
  \varphi_w\colon \CH_{\ssleq w}\to K^{\Gbar}(Z_{\ssleq w})
  \quad\text{by}\quad \varphi_w (\theta_\lambda T_y)= (d_1)_*
  q^*[\CL_\lambda] \star_w (j_y^w)_* (\xi_y)
  \]
  for $\lambda$ in $X(T)$ and $y$ in $W$ with $y\leq w$.

  Set $d=j_1d_1\colon \FNt \to Z$, so $d(x,gB)=(x, gB, gB)$. By the
  definition of $\varphi$, for $\lambda$ in $X(T)$, $\varphi(
  \theta_\lambda)= d_* q^*[\CL_\lambda]$. Thus, using
  equation~\eqref{eq:star1} and~\autoref{lem:starlin} we have
  \begin{align*}
    (j_w)_* \varphi_w (\theta_\lambda T_y) & = (j_w)_* \big((d_1)_*
    q^*[\CL_\lambda] \star_w (j_y^w)_* \xi_y\big) \\
    & = (j_1)_*(d_1)_* q^*[\CL_\lambda] \star (j_w)_*(j_y^w)_* \xi_y 
    = d_* q^*[\CL_\lambda] \star (j_y)_* \xi_y \\ & = \varphi(\theta_\lambda)
    \star \varphi(T_y) = \varphi(\theta_\lambda T_y) .
  \end{align*}
  This proves the first statement.

  To prove~\autoref{it:phi2}, let $\BBC_{Z_w}$ be the trivial line bundle on
  $Z_w$ and let
  \[
  p_{w,1}\colon Z_w\to \FNt\quad\text{by}\quad p_{w,1} (x,gB, gwB)= (x, gB).
  \]
  Then $q_{w,1}=q p_{w,1}$. Set $V_w=(\Ztilde_1\times \FNt) \cap (\FNt
  \times \Ztilde_{ w})$ and let $p_{12}'\colon V_w \to Z_1$ and
  $p_{13}'\colon V_w \to Z_w$ be the obvious projections.  It is
  straightforward to check that $p_{w,1}$ and $p_{12}'$ are smooth,
  that $p_{13}'$ is an isomorphism, and that the diagram
  \[
  \xymatrix{ Z_w \ar[r]^-{(p_{13}')\inverse} \ar[d]^{p_{w,1}} & V_w
    \ar[d]^{p_{12}'} \\ 
    \FNt \ar[r]^-{d_1} &Z_1}
  \]
  is cartesian. Thus $(p_{12}')^* (d_1)_*= (p_{13}')\inverse_* p_{w,1}^*$.
  Using the $K^{\Gbar}(Z_1)$-module structure $\star_w'$ on
  $K^{\Gbar}(Z_{w})$, we have
  \begin{align*} 
    r_w^* \varphi_w(\theta_\lambda T_w)&= r_w^* \big( (d_1)_*
    q^*[\CL_\lambda] \star_w \xi_w \big)&& \\
    &= (d_1)_* q^*[\CL_\lambda] \star_w' r_w^*( \xi_w) &&
    \text{\autoref{lem:starlin}}  \\
    &= (d_1)_* q^*[\CL_\lambda] \star_w' \epsilon_w [\BBC_{Z_w}] &&
    \text{\cite[\S8.9]{lusztig:bases}}  \\
    &= \epsilon_w\ (p_{13}')_*\, (p_{12}')^* (d_1)_* q^*[\CL_\lambda] &&
    \text{\autoref{pro:convlin}\,\autoref{i:conv3}}  \\
    &= \epsilon_w\ p_{w,1}^*\, q^*[\CL_\lambda] && \\
    &= \epsilon_w\ q_{w,1}^*\, [\CL_\lambda]. &&
  \end{align*}
  This completes the proof of the second statement.

  The last statement in the theorem follows from~\autoref{it:phi2} and the
  fact that $q_{w,1}^*$ is an isomorphism.

  Choose a linear order on the interval $[1,w]$ in the Bruhat poset of $W$
  that extends the Bruhat order. This linear order determines gradings on
  $\CH_{\ssleq w}$ and $K^{\Gbar}(Z_{\ssleq w})$. As $\{\,
  \theta_\lambda T_y\mid \lambda\in X(T) \,\}$ is an $A$-basis of $\CH_{y}$
  and $\{\, r_y^* \varphi_y(\theta_\lambda T_y)\mid \lambda\in X(T) \,\}$ is
  an $A$-basis of $K^{\Gbar}(Z_y)$ for $y$ in $W$, the associated graded map
  $\gr \varphi_w$ is an isomorphism. Thus, $\varphi_w$ is an
  isomorphism. \qed
\end{proof}

It follows from the theorem that for $z$ in $\WIJ$,
\[
\etaIJ_* \varphi(\CH_{\ssleq z})= \etaIJ_* (j_z)_*( K^{\Gbar}(Z_{\ssleq z})
)= (j^{IJ}_z)_* \eta^{\ssleq z}_*( K^{\Gbar}(Z_{\ssleq z})) \subseteq
(j^{IJ}_z)_* ( K^{\Gbar}(\XIJ_{\ssleq z})).
\]
On the other hand, we have seen that $\etaIJ_* \varphi =\psi^{IJ} \chi^{IJ}$
and that $\chi^{\ssleq z}$ is surjective, so
\[
\etaIJ_* \varphi(\CH_{\ssleq z})=\psi^{IJ} \chi^{IJ}(\CH_{\ssleq z}) =
\psi^{IJ} \chi^{\ssleq z} (\CH_{\ssleq z})= \psi^{IJ} (\CHIJ_{\ssleq z}).
\]
Therefore, $\psi^{IJ} (\CHIJ_{\ssleq z}) \subseteq (j^{IJ}_z)_* (
K^{\Gbar}(\XIJ_{\ssleq z}))$ and so there is an $A$-module
homomorphism $\psi_z\colon \CHIJ_{\ssleq z}\to K^{\Gbar}(\XIJ_{\ssleq
  z})$ such that $\eta^{\ssleq z}_* \varphi_z= \psi_z \chi^{\ssleq
  z}$. In particular, the maps $\varphi_z$ and $\psi_z$ in the
left-hand square in~\eqref{eq:wideiso} are defined and the diagram
commutes.

\subsection{Proof of~\autoref{thm:wideiso2}} \label{ssec:isoz3}

In this subsection we complete the proof of~\autoref{thm:wideiso2}. The
arguments above show that diagram~\eqref{eq:wideiso} commutes.

To prove~\autoref{thm:wideiso2}\,\autoref{i:wi2}, let $f_1\colon \CH_{\ssleq
  z} \to K^{\Lbar_z}(\SNt_z)$ be the composition across the top row
in~\eqref{eq:wideiso}. Then $f_1= i_z^* \res_z r_z^* \varphi_z$, where
$i_z^*$, $\res_z$, and $\varphi_z$ are isomorphisms and $r_z^*$ is
surjective. If $y<z$, then $\varphi_w(\CH_y)\subseteq K^{\Gbar}(Z_{\ssleq
  y})$ and so it follows from the exactness of the top row of
diagram~\eqref{eq:exwz} that $\varphi_z(\CH_y)$ is in the kernel of
$r_z^*$. Therefore, $\CH_y$ is in the kernel of $f_1$.

By~\autoref{thm:basis}\,\autoref{it:phi2}, $r_z^* \varphi_z(\theta_\lambda
T_z)=\epsilon_z\ q_{z,1}^* ([\CL_\lambda])$, and so
\[
f_1(\theta_\lambda T_z)=\epsilon_z\ i_z^* \res_z \ q_{z,1}^*
([\CL_\lambda]).
\]  
Let $g_z\colon F\to Z_z$ be the inclusion and let $f_z\colon L_z/B_z\to G/B$
be the canonical embedding. Then the diagram
\[
\xymatrix{\SNt_z\ar[r]^-{i_z} \ar[d]^{q_z} & F \ar[r]^-{g_z} &
  Z_z\ar[d]^{q_{z,1}} \\
  L_z/B_z \ar[rr]^{f_z} && G/B }
\]
commutes. The restriction map $\res_z\colon K^{\Gbar}(Z_z)\to
K^{\Lbar_z}(F)$ is induced by the functor $g_z^*$ from $\Gbar$-equivariant
coherent sheaves on $Z_z$ to $\Lbar_z$-equivariant sheaves on
$F$. Therefore,
\[
i_z^* \res_z \ q_{z,1}^* ([\CL_\lambda])= [ i_z^* g_z^* q_{z,1}^*
(\CL_\lambda)] = [q_z^* f_z^* (\CL_\lambda)] = [q_z^* (\CL_\lambda^z)]
=q_z^* [\CL_\lambda^z]
\]
and so $f_1(\theta_\lambda T_z)=\epsilon_z q_z^* [\CL_\lambda^z]$.

To prove~\autoref{thm:wideiso2}\,\autoref{i:wi3}, let $f_2\colon
\CHIJ_{\ssleq z} \to K^{\Lbar_z}(\SN_z)$ be the composition across the
bottom row in~\eqref{eq:wideiso}. Then $f_2= (i_z^{IJ})^* \res_z^{IJ}
(r_z^{IJ})^* \psi_z$ and $f_2 \chi^{\ssleq z}= (p_z)_* f_1$.  If $y$ is in
$\WIJ$ and $y<z$, then $\psi_z(\CHIJ_y)\subseteq K^{\Gbar}(\XIJ_{\ssleq y})$
and so it follows from the exact sequence~\eqref{eq:exz} that
$\psi_z(\CHIJ_y)$ is in the kernel of $(r_z^{IJ})^*$. Therefore, $\CHIJ_y$
is in the kernel of $f_2$. Suppose $\lambda$ is in $X(T)$. Then
\[
f_2(C_{w_I} \theta_\lambda T_zC_{w_J})= f_2 \chi^{\ssleq z}(
\theta_\lambda T_z)= (p_z)_* f_1 ( \theta_\lambda T_z)= \epsilon_z (p_z)_*
q_z^*([\CL^z_\lambda]).
\]
As observed before the statement of~\autoref{thm:wideiso2}, $\{\, C_{w_I}
\theta_\lambda T_zC_{w_J} \mid \lambda \in X(T)^+_z\,\}$ is a $A$-basis of
$\CHIJ_z$, and by~\autoref{lem:ostlem}, $\{\, \epsilon_z (p_z)_*
q_z^*([\CL^z_\lambda])\mid \lambda\in X(T)^+_z \,\}$ is an $A$-basis of
$K^{\Lbar_z}(\SN_z)$. It follows that $f_2$ is an isomorphism.

Finally, to prove~\autoref{thm:wideiso2}\,\autoref{i:wi1} it remains to show
that $\eta^{\ssleq z}_*$ is surjective and that $\psi_{z}$ is an
isomorphism. The fact that $\eta^{\ssleq z}_*$ is surjective was established
in~\autoref{cor:surj}. Since $\psi_z \chi^{\ssleq z}= \eta^{\ssleq z}_*
\varphi_z$, we see that $\psi_z$ is a surjection as well. Induction on the
length of $z$ and statement~(\ref{i:wi3}) of the theorem show that
\[
\{\, \psi_z(C_{w_I} \theta_\lambda T_y C_{w_J}) \mid y\in \WIJ,\ y\leq
z,\ \lambda\in X(T)_{y}^+ \,\}
\] 
is a linearly independent subset of $K^{\Gbar}(\XIJ_{\ssleq z})$ and so
$\psi_z$ is an injection. This completes the proof
of~\autoref{thm:wideiso2}.

\bigskip\noindent {\bf Acknowledgments}: This work was partially supported
by a grant from the Simons Foundation (Grant \#245399 to J. Matthew
Douglass).  The authors acknowledge the financial support of a DFG grant for
the enhancement of bilateral cooperation and the DFG-priority program
SPP1388 ``Representation Theory.''  The authors would like to thank Victor
Ostrik for clarifying the assertions in \cite{ostrik:equivariant}, William
Graham and an anonymous referee for help with the proof
of~\autoref{pro:special}, and the same anonymous referee for numerous
suggestions that substantially improved the paper.




\end{document}